\documentclass[12pt]{article}
\usepackage{amsmath,amssymb,amsthm, amsfonts}
\usepackage{etaremune, enumerate, float, verbatim}
\usepackage{hyperref}
\usepackage{graphicx}
\usepackage{url}
\usepackage{dsfont} 
\usepackage{tikz, graphicx, subcaption, caption}
\usepackage[algo2e,ruled,vlined]{algorithm2e}
\usepackage{mathtools}
\DeclarePairedDelimiter{\ceil}{\lceil}{\rceil}

\usepackage{color}
\definecolor{red}{rgb}{1,0,0}
\def\red{\color{red}}
\definecolor{blue}{rgb}{0,0,.7}

\definecolor{green}{rgb}{0,.6,0}

\definecolor{purp}{rgb}{.5,0,.5}

\numberwithin{figure}{section}   
\usepackage{csquotes} 
\usepackage{comment} 

\usepackage{tabularray}

\newtheorem{thm}{Theorem}[section]
\newtheorem{cor}[thm]{Corollary}
\newtheorem{lem}[thm]{Lemma}
\newtheorem{prop}[thm]{Proposition}

\newtheorem{Model}{Model}
\newtheorem{definition}[thm]{Definition}

\newcommand{\bea}{\begin{eqnarray*}} 
\newcommand{\eea}{\end{eqnarray*}}

\theoremstyle{definition}

\theoremstyle{definition}

\newcommand{\thr}{\operatorname{th_{dim}}}
\newcommand{\thre}{\operatorname{th_{edim}}}
\newcommand{\thrm}{\operatorname{th_{mdim}}}
\newcommand{\thrx}{\operatorname{th_{xdim}}}
\newcommand{\dist}{\operatorname{dist}}
\newcommand{\ths}{\operatorname{th_s}}

\newcommand{\edim}{\operatorname{edim}}
\newcommand{\mdim}{\operatorname{mdim}}
\newcommand{\xdim}{\operatorname{xdim}}
\newcommand{\D}{\mathcal{D}}

\def\Circ#1#2{\operatorname{Circ}_{#1}\!\left(#2\right)}

\newcommand{\paren}[1]{\left(#1\right)} 
\newcommand{\cl}[1]{\left\lceil#1\right\rceil} 

\title{Throttling for metric dimension and its variants}
\author{Boris Brimkov\thanks{Department of Mathematics, Statistics, and Physics, Slippery Rock University (boris.brimkov@sru.edu)} \\\and
Peter Diao\thanks{pdiao@alumni.princeton.edu} \\\and
  Jesse Geneson\thanks{Department of Mathematics and Statistics, San José State University (geneson@gmail.com)} \\\and
  Carolyn Reinhart\thanks{Department of Mathematics and Statistics, Swarthmore College (creinha1@swarthmore.edu)} \\\and
  Shen-Fu Tsai\thanks{Department of Mathematics, National Central University (parity@gmail.com)} \\\and
  William Wang\thanks{Massachusetts Institute of Technology (william12wang@gmail.com)} \\\and
  Kyle Worley\thanks{Department of Mathematics and Statistics, San José State University (kyle.worley@sjsu.edu)} \\
}
\date{}
\usepackage[left]{lineno}
\begin{document}
\maketitle
\begin{abstract}
Metric dimension is a graph parameter that has been applied to robot navigation and finding low-dimensional vector embeddings. Throttling entails minimizing the sum of two available resources when solving certain graph problems. In this paper, we introduce throttling for metric dimension, edge metric dimension, and mixed metric dimension. In the context of vector embeddings, metric dimension throttling finds a low-dimensional, low-magnitude embedding with integer coordinates. We show that computing the throttling number is NP-hard for all three variants. We give formulas for the throttling numbers of special families of graphs, and characterize graphs with extremal throttling numbers. We also prove that the minimum possible throttling number of a graph of order $n$ is $\Theta\left(\frac{\log{n}}{\log{\log{n}}}\right)$, while the minimum possible throttling number of a tree of order $n$ is $\Theta(n^{1/3})$ or $\Theta(n^{1/2})$ depending on the variant of metric dimension. \\

\noindent \textbf{Keywords:} Metric dimension; throttling; edge metric dimension; mixed metric dimension
\end{abstract}

\section{Introduction}
Metric dimension is a graph parameter used for robot navigation, low-dimensional vector embeddings, and classification of chemical compounds \cite{harary, kuziak, slater, tillquist}. In the first application, a robot is sent to a planet to perform a task such as surveillance or locating an object. A graph has been drawn on the surface of the planet, with vertices representing locations on the surface and edges representing pairs of locations that the robot can travel between. Several landmarks have been constructed on the surface of the planet to help with navigation. The robot is dropped somewhere on the planet, but its exact location is unknown. However, the landmarks have been placed in such a way that the robot can determine its exact position if it knows the distance to each landmark, regardless of where it was dropped on the planet.

A \emph{resolving set} of $G$ is a set of vertices $S \subset V(G)$ such that for every $x, y \in V(G)$, there exists some $v \in S$ such that $d(v, x) \neq d(v, y)$. The standard \emph{metric dimension} $\dim(G)$ is defined as the minimum possible size of a resolving set for $G$. In other words, if $G$ is the graph that was drawn on the surface of the planet, then $\dim(G)$ is the minimum possible number of landmarks necessary for the robot to determine its exact position, regardless of where it was dropped. Much is known about $\dim(G)$, including the exact values for trees and $d$-dimensional grids \cite{khuller} and characterizations of the graphs with minimum and maximum possible metric dimension \cite{chartrand}. 
For graphs $G$ with diameter at most $D$ and metric dimension at most $k$, there are sharp bounds on the maximum possible order of $G$ in terms of $D$ and $k$~\cite{hernando}. For graphs with bounded metric dimension, there are sharp bounds on the maximum order of subgraphs of bounded diameter \cite{gen, gkl, khuller}.


In addition to the standard metric dimension $\dim(G)$, many variants have been studied. In one variant called \emph{edge metric dimension}, any pair of edges in the graph must be distinguished by some vertex in the resolving set \cite{kelenc}, where we define $\dist(e, v) = \min(\dist(u, v), \dist(w, v))$ when $e = \{u, w\}$. An \emph{edge resolving set} of $G$ is a set of vertices $S \subset V(G)$ such that for every $e, f \in E(G)$, there exists some $v \in S$ such that $\dist(v, e) \neq \dist(v, f)$. The \emph{edge metric dimension} $\edim(G)$ is defined as the minimum possible size of an edge resolving set for $G$. In the context of the robot determining its position, the edge metric dimension is analogous to the metric dimension, except the robot is dropped on an edge instead of a vertex. Several papers on edge metric dimension \cite{gen, gkl, wei, zhu, zubrilina2, zubrilina} have obtained characterization results, extremal bounds, pattern avoidance results, asymptotic bounds for random graphs, and exact values for special families of graphs.

Another variant is the \emph{mixed metric dimension}, in which the set of objects to be distinguished includes both the vertices of $G$ as well as the edges \cite{mmd}. In particular, a \emph{mixed resolving set} of $G$ is a set of vertices $S \subset V(G)$ such that for every $e, f \in E(G)$ and $u, w \in V(G)$, there exists some $v \in S$ such that $\dist(v, e) \neq \dist(v, f)$, there exists some $v' \in S$ such that $\dist(v',e) \neq \dist(v',u)$, and there exists some $v'' \in S$ such that $\dist(v'',u) \neq \dist(v'',w)$. The \emph{mixed metric dimension} $\mdim(G)$ is defined as the minimum possible size of a mixed resolving set for $G$. From the definitions, it follows that $\mdim(G) \ge \max(\dim(G), \edim(G))$ for all graphs $G$. In the context of the robot determining its position, the mixed metric dimension is analogous to the metric dimension and the edge metric dimension, except the robot can be dropped on both vertices and edges. There have been several papers on mixed metric dimension \cite{danas, mmd_hyper, sedlar1, sedlar2} which have obtained characterization results, extremal bounds, and exact values for families of graphs.


In this paper, we introduce an extension of the above variants of metric dimension in which the goal is to minimize the number of landmarks plus the search radius of the sensor. 
The motivation is that each landmark has a cost $\ell$ and the sensor has a cost $c(r)$ which increases with its detection radius $r$. If we use $k$ landmarks and a sensor with detection radius $r$, then the total cost will be $c(r)+k \ell$. We focus on the subproblem where $c(r) = r$ and $\ell = 1$. The problem of minimizing $r+k$ is called \textit{throttling}. Throttling is usually thought of as minimizing a sum of resources and time. In our context, $c(r) = r$ would be the time cost, where the sensor takes $r$ time units to observe landmarks at distance $r$, and $k$ would be the resource cost.

The notion of throttling was originally introduced by Butler and Young \cite{butler} for the graph coloring process \emph{zero forcing} as the minimum possible sum of the number of initially colored vertices and the time for the graph to be completely colored. 
Like metric dimension, zero forcing has multiple variants including positive semidefinite zero forcing~\cite{psd}, skew zero forcing~\cite{skew}, and probabilistic zero forcing~\cite{cong}. Throttling has been defined and investigated for all of these variants \cite{psdth, skewth, ghpzf}. Throttling has also been studied for the cop versus robber game \cite{cr2, cr1, damage}, the cop versus gambler game \cite{gth2, gth1}, and the power domination problem \cite{pdthrot}. 

In our definition of throttling for metric dimension, we assume that the robot has an $r$-sensor that can determine the distance to all landmarks within a bounded radius $r$. In other words, if the robot is at vertex $x$ and $v$ is any landmark in the graph, the sensor tells the robot the value of $\dist_r(x, v) = \min(\dist(x, v), r+1)$. 

Define a \emph{distance-$r$ resolving set} of $G$ to be a set of vertices $S \subset V(G)$ such that for every $x, y \in V(G)$, there exists some $v \in S$ such that $\dist_r(v, x) \neq \dist_r(v, y)$. The \emph{$r$-truncated metric dimension} $\dim_r(G)$ is the minimum possible size of a distance-$r$ resolving set of $G$~\cite{frongillo,gt25}. In the case that $r = 1$, a distance-$r$ resolving set is also called an \emph{adjacency resolving set}, and the $r$-truncated metric dimension is also called the \emph{adjacency dimension} \cite{bermudo, jannesari}. 

\begin{definition}
    The metric dimension throttling number of $G$, $\thr(G)$, is the minimum possible value of $\dim_r(G)+r$ over all positive integers $r$.
\end{definition}


In the context of vector embeddings, the throttling number minimizes the sum of the dimension of the embedding and the maximum possible value of any coordinate in the embedding. Thus, a minimum throttling configuration provides an injective vector embedding with integer coordinates in which the number of dimensions and the magnitudes of the vectors are low. This is desirable, e.g., when transforming sequences of nucleotides or amino acids into numeric vectors as input for machine learning algorithms \cite{tillquist}.

It is clear that the maximum possible value of $\thr(G)$ over graphs $G$ of order $n$ is $n-1$, since we can place $n-1$ landmarks on the graph so that the robot knows its initial vertex as soon as it is dropped. Moreover, an example of a graph $G$ of order $n$ with $\thr(G) = n-1$ is $K_n$, since in $K_n$ every pair of vertices have the same neighborhood and thus there can be no pair of vertices that do not both have landmarks. The minimum possible value of $\thr(G)$ is a more interesting question. In this paper, we show that the minimum possible value of $\thr(G)$ over all graphs $G$ of order $n$ is $\Theta(\frac{\log{n}}{\log{\log{n}}})$. We also obtain $\Theta(\sqrt{n})$ bounds on metric dimension throttling numbers for families of graphs such as paths, cycles, spiders, and circulant graphs, as well as sharp bounds for grids and exact values for complete bipartite graphs.

In addition to standard metric dimension, we also define throttling for variants of metric dimension. We prove some general throttling results that apply to arbitrary subset-variants of metric dimension, and we also focus in particular on throttling for edge metric dimension and mixed metric dimension. 

As with metric dimension throttling, let $\dist_r(e, v) = \min(\dist(e, v), r+1)$. Define a \emph{distance-$r$ edge resolving set} of $G$ to be a set of vertices $S \subset V(G)$ such that for every $e, f \in E(G)$, there exists some $v \in S$ such that $\dist_r(v, e) \neq \dist_r(v, f)$. We define \emph{$r$-truncated edge metric dimension} $\edim_r(G)$ to be the minimum possible size of a distance-$r$ edge resolving set of $G$.

\begin{definition}
    The edge metric dimension throttling number of $G$, $\thre(G)$, is the minimum possible value of $\edim_r(G)+r$ over all positive integers $r$.
\end{definition}

Similarly, define a \emph{distance-$r$ mixed resolving set} of $G$ to be a set of vertices $S \subset V(G)$ such that for every $e, f \in E(G)$ and $u, w \in V(G)$, there exists some $v \in S$ such that $\dist_r(v, e) \neq \dist_r(v, f)$, there exists some $v' \in S$ such that $\dist_r(v',e) \neq \dist_r(v',u)$, and there exists some $v'' \in S$ such that $\dist_r(v'',u) \neq \dist_r(v'',w)$. We define \emph{$r$-truncated mixed metric dimension} $\mdim_r(G)$ to be the minimum possible size of a distance-$r$ mixed resolving set of $G$.

\begin{definition}
    The mixed metric dimension throttling number of $G$, $\thrm(G)$, is the minimum possible value of $\mdim_r(G)+r$ over all positive integers $r$.
\end{definition}

Based on the preceding definitions, it is natural to extend the notion of throttling to other variants of metric dimension. We first define the class of subset-variants of metric dimension, which encompasses all the variants discussed above, as follows.

\begin{definition}
Let $G$ be a graph and $T_{\xdim}(G)$ be a set of subsets of $V(G)$. For any $X \subseteq V(G)$, let \[\dist(X, v) = \min_{u \in X} \dist(u, v),\] where if $X = \emptyset$ then $\dist(X, v) = \infty$. A set of vertices $S \subseteq V(G)$ is $\xdim$-resolving for $G$ if for all distinct $X, Y \in T_{\xdim}(G)$ there exists $v \in S$ such that $\dist(X, v) \neq \dist(Y, v)$. The minimum possible size of an $\xdim$-resolving set for $G$ is $\xdim(G)$, and we say that $\xdim$ is a \emph{subset-variant} of metric dimension. 
\end{definition}

Standard metric dimension, edge metric dimension, and mixed metric dimension are all subset-variants of metric dimension. In particular, with slight abuse of notation, when $\xdim = \dim$, we have $T_{\xdim}(G) = V(G)$, when $\xdim = \edim$, we have $T_{\xdim}(G) = E(G)$, and when $\xdim = \mdim$, we have $T_{\xdim}(G) = V(G) \cup E(G)$. 

We now define throttling for an arbitrary subset-variant of metric dimension. For any subset $S \subseteq V(G)$ and $v \in V(G)$, let $\dist_r(S, v) = \min(\dist(S, v), r+1)$. For any subset-variant $\xdim$ of metric dimension, define a \emph{distance-$r$ $\xdim$-resolving set} of $G$ to be a set of vertices $S \subset V(G)$ such that for all distinct $X, Y \in T_{\xdim}(G)$, there exists $v \in S$ such that $\dist_r(X, v) \neq \dist_r(Y, v)$. We define \emph{$r$-truncated $\xdim$}, denoted $\xdim_r(G)$, to be the minimum possible size of a distance-$r$ $\xdim$-resolving set of $G$.

\begin{definition}
    For any subset-variant $\xdim$ of metric dimension, the $\xdim$ throttling number of $G$, $\thrx(G)$, is the minimum possible value of $\xdim_r(G)+r$ over all positive integers $r$.
\end{definition}

This paper is organized as follows. Section \ref{prelim} recalls graph-related terminology and notations. In Section~\ref{sec:complexity}, we show that it is NP-hard to compute the throttling numbers for standard, edge, and mixed metric dimension, and give an integer programming model for computing the throttling number.
In Section~\ref{sec:general}, we prove general results about throttling for arbitrary subset-variants of metric dimension, which we use to derive corollaries for standard, edge, and mixed metric dimension. 
In Section~\ref{sec:standard}, we focus on standard metric dimension; we derive asymptotic bounds, characterize graphs with extremal throttling numbers, and study throttling for specific families of graphs.
In Sections~\ref{sec:edge} and~\ref{sec:mixed}, we obtain analogous results about edge and mixed metric dimension, respectively. 
We conclude with a summary and directions for future work in Section~\ref{sec:conclusion}.

\section{Terminology}
\label{prelim}

A simple graph $G=(V,E)$ consists of a vertex set $V$ and an edge set $E$ of two-element subsets of $V$.  The \emph{order} of $G$ is denoted by $n=|V|$. Two vertices $v,w\in V$ are \emph{adjacent}, or \emph{neighbors}, if $\{v,w\}\in E$. When there is no scope for confusion, we will write the edge $e=\{u,v\}$ as $uv$. The \emph{neighborhood} of $v\in V$ is the set of all vertices which are adjacent to $v$, denoted $N(v)$; the \emph{closed neighborhood} of $v$, denoted $N[v]$, is the set $N(v)\cup\{v\}$. The \emph{degree} of $v\in V$ is defined as $\deg(v)=|N(v)|$. Given $S \subset V$, the \emph{induced subgraph} $G[S]$ is the subgraph of $G$ whose vertex set is $S$ and whose edge set consists of all edges of $G$ which have both endpoints in $S$. The \emph{complement} of graph $G$, denoted $\overline{G}$, is the graph with vertex set $V(G)$, where two vertices $u$ and $v$ are adjacent in $\overline{G}$ if and only if they are not adjacent in $G$. The \emph{disjoint union} of graphs $G$ and $H$ is denoted $G + H$. 

A graph is a \emph{singleton} if it consists of a single vertex of degree $0$. The \emph{degree set} of a graph $G$ is the set of integers that are the degrees of the vertices of the graph. For a set of integers $S$, a {\em circulant graph} $\Circ n S$ is a graph on vertices $\{1,\dots,n\}$ with $ij$ being an edge if and only if $i-j$ or $j-i$ is in $S$ modulo $n$. We call $S$ the {\em connection set} and by convention we assume for $x\in S$, $1\leq x\leq \lfloor \frac{n}{2}\rfloor$. A \emph{spider} is a tree that has only one vertex of degree greater than $2$. This vertex is called the \emph{body vertex}, and the graph obtained by removing the body vertex is a disjoint union of paths. Each of these paths is called a \emph{leg} of the spider, and the number of edges of each of these paths is the length of the corresponding leg. We say that a spider is \emph{balanced} if all legs have the same length, otherwise the spider is \emph{unbalanced}. 

The \textit{distance vector} of a vertex $v$ with respect to a list of landmarks $L_1, L_2, \dots, L_k$ is the list $\dist(v,L_1),\dist(v,L_2),\dots,\dist(v,L_k)$. Thus, a set of vertices $S \subseteq V(G)$ is $\xdim$-resolving for $G$ if all $X \in T_{\xdim}(G)$ have a unique distance vector with respect to some ordering of $S$. Note that distances are $\infty$ for vertices not in the same connected component. The \textit{$r$-truncated distance vector} of a vertex $v$ with respect to a list of landmarks $L_1, L_2, \dots, L_k$ is the list $\min(r+1,\dist(v,L_1)),\min(r+1,\dist(v,L_2)),\dots,\min(r+1,\dist(v,L_k))$. For other graph theoretic notions and notations, we generally follow \cite{west}. 

\section{Complexity and computation}\label{sec:complexity}

Let {\sc Metric Dimension} (MD) denote the decision problem of determining for a given graph $G$ and positive integer $k$ whether $\dim(G) \le k$. Let {\sc Metric Dimension Throttling} (MDT) denote the decision problem of determining for a given graph $G$ and positive integer $k$ whether $\thr(G) \le k$. Analogously define the decision problems {\sc Edge Metric Dimension} (EMD), {\sc Edge Metric Dimension Throttling} (EMDT), {\sc Mixed Metric Dimension} (MMD), {\sc Mixed Metric Dimension Throttling} (MMDT), {\sc Subset-variant Metric Dimension} (SMD), and {\sc Subset-variant Metric Dimension Throttling} (SMDT). It is already known that MD is NP-Complete \cite{khuller}, EMD is NP-Complete \cite{kelenc}, and MMD is NP-Complete \cite{mmd}. Therefore SMD is NP-Complete because MD, EMD, and MMD can all be reduced to SMD.
We will now show that MDT, EMDT, and MMDT are all also NP-Complete. 

\begin{thm}
    {\sc Metric Dimension Throttling} (MDT) is NP-Complete. 
\end{thm}

\begin{proof}
    It is clear that MDT is in NP. To see that it is NP-hard, we exhibit a polynomial-time reduction from MD to MDT. For any graph $G$ of order $n$, let $G'$ be obtained from $G$ by adding $n$ disjoint paths, each with $n+1$ vertices, as well as a singleton vertex. 
    
    By construction, we must have $\thr(G') = 2n+\dim(G)$. Indeed, any resolving set for $G$ of size $\dim(G)$ can be extended to a resolving set for $G'$ of size $n+\dim(G)$ by putting a new landmark on a single endpoint from each path, and the greatest distance between any two vertices in the same connected component is $n$, so $\thr(G') \le 2n+\dim(G)$. 

    To see that $\thr(G') \ge 2n+\dim(G)$, let $\thr(G')=r+\dim_r(G')$ and consider a minimum $r$-resolving set of $G'$. We split into two cases. For the first case, suppose that we $r$-resolve $G'$ with at least two landmarks on each of the $n$ disjoint paths. Then, we use at least $2n$ landmarks for the disjoint paths, and we must use at least $\dim(G)$ landmarks to resolve the vertices of $G$ in $G'$, so the $r$-resolving set has size at least $2n+\dim(G)$. 
    
    For the second case, suppose that at least one of the disjoint paths does not have two landmarks.  First, note that every disjoint path must have at least one landmark, or it cannot be resolved. Note that any disjoint path with a single landmark must have the landmark at an endpoint, or else the landmark would not resolve the vertices of the path. Moreover, the throttling radius $r$ must be at least $n$, or else the endpoint that is opposite from the single landmark would not be resolved from the singleton. Therefore, the $r$-resolving set has size at least $n+\dim(G)$, and the throttling radius $r$ is at least $n$, so we have $\thr(G') \ge 2n+\dim(G)$.
    
    Thus, we have $\dim(G) \le k$ if and only if $\thr(G') \le 2n+k$. Therefore, MDT is NP-hard.
\end{proof}

A modified construction is required in the NP-hardness reduction for edge metric dimension throttling.

\begin{thm}
    {\sc Edge Metric Dimension Throttling} (EMDT) is NP-Complete. 
\end{thm}

\begin{proof}
    It is clear that EMDT is in NP. To see that it is NP-hard, we exhibit a polynomial-time reduction from EMD to EMDT. For any graph $G$ of order $n$, let $G'$ be obtained from $G$ by adding $n$ disjoint paths, each with $n+2$ vertices, as well as a copy of $K_2$. 

    By construction, we must have $\thre(G') = 2n+\edim(G)$. Indeed, any edge resolving set for $G$ of size $\edim(G)$ can be extended to an edge resolving set for $G'$ of size $n+\dim(G)$ by putting a new landmark on a single endpoint from each path, and the greatest distance between any vertex and any edge in the same connected component is $n$, so $\thre(G') \le 2n+\edim(G)$. 

    To see that $\thre(G') \ge 2n+\edim(G)$, we split into two cases. For the first case, suppose that we resolve the edges of $G'$ with at least two landmarks on each of the $n$ disjoint paths. Then, we use at least $2n$ landmarks for the disjoint paths, and we must use at least $\edim(G)$ landmarks to resolve the edges of $G$ in $G'$, so the edge resolving set has size at least $2n+\edim(G)$.

    For the second case, suppose that at least one of the disjoint paths does not have two landmarks.  First, note that every disjoint path must have at least one landmark, or its edges cannot be resolved. Note that any disjoint path with a single landmark must have the landmark at an endpoint, or else the landmark would not resolve the edges of the path. Moreover, the throttling radius must be at least $n$, or else the edge that contains the endpoint that is opposite from the single landmark would not be resolved from the edge in the copy of $K_2$. Therefore, the edge resolving set has size at least $n+\edim(G)$, and the throttling radius is at least $n$, so we have $\thre(G') \ge 2n+\edim(G)$.
    
    Thus, we have $\edim(G) \le k$ if and only if $\thre(G') \le 2n+k$. Therefore, EMDT is NP-hard.
\end{proof}

We use a similar reduction to obtain NP-hardness of throttling for mixed metric dimension. In this case, note that each path must have at least two landmarks. 

\begin{thm}
    {\sc Mixed Metric Dimension Throttling} (MMDT) is NP-Complete. 
\end{thm}

\begin{proof}
    It is clear that MMDT is in NP. To see that it is NP-hard, we exhibit a polynomial-time reduction from MMD to MMDT. For any graph $G$ of order $n$, let $G'$ be obtained from $G$ by adding $n$ disjoint paths, each with $n+2$ vertices. 

    By construction, we must have $\thrm(G') = 3n+\mdim(G)$. Indeed, any mixed resolving set for $G$ of size $\mdim(G)$ can be extended to a mixed resolving set for $G'$ of size $2n+\dim(G)$ by putting a new landmark on both endpoints for each path. The greatest distance between any vertex and any vertex or edge in the same connected component is $n+1$, so $\thrm(G') \le 3n+\mdim(G)$. 

    To see that $\thrm(G') \ge 3n+\mdim(G)$, we split into two cases. For the first case, suppose that we resolve the vertices and edges of $G'$ with at least $3$ landmarks on each of the $n$ disjoint paths. Then, we use at least $3n$ landmarks for the disjoint paths, and we must use at least $\mdim(G)$ landmarks to resolve the vertices and edges of $G$ in $G'$, so the mixed resolving set has size at least $3n+\mdim(G)$.

    For the second case, suppose that at least one of the disjoint paths does not have three landmarks. First, note that every disjoint path must have at least two landmarks, or its vertices and edges cannot be resolved. Also, note that any disjoint path with two landmarks must have the landmarks at both endpoints, or else the landmarks would not resolve the vertices and edges of the path \cite{mmd}. Moreover, the throttling radius must be at least $n$, or else any edge that contains an endpoint of one of the disjoint paths would not be resolved from the endpoint that it contains. Therefore, the mixed resolving set has size at least $2n+\mdim(G)$, and the throttling radius is at least $n$, so we have $\thrm(G') \ge 3n+\mdim(G)$.
    
    Thus, we have $\mdim(G) \le k$ if and only if $\thrm(G') \le 3n+k$. Therefore, MMDT is NP-hard.
\end{proof}

Since the problems above are subproblems of SMDT, we have the following corollary.

\begin{cor}
{\sc Subset-variant Metric Dimension Throttling} (SMDT) is NP-Complete.    
\end{cor}

We conclude this section by presenting an algorithm based on integer programming for computing the metric dimension throttling number of a graph. We start with a model for truncated metric dimension. 

\begin{Model}{IP model for truncated metric dimension}\label{model1}
\begin{align}
\nonumber \min \quad &  \sum_{i=k}^n  x_k \\
\nonumber \emph{s.t.:} \quad &  \sum_{k = 1}^n \left|\min(\dist(v_k,v_i),r+1)-\min(\dist(v_k,v_j),r+1)\right|x_k > 0\\
\nonumber & \qquad\qquad\qquad\qquad\qquad\qquad\qquad\qquad\qquad \emph{for } 1 \le i < j \le n\\
\nonumber & x_i\in \{0,1\} \emph{ for  } 1\leq k\leq n. 
\end{align}
\end{Model}

\begin{prop}
The optimal value of Model \ref{model1} is equal to  $\dim_r(G)$.
\end{prop}
\begin{proof}
Let $G$ be a graph of order $n$ with vertices $v_1, \dots, v_n$, where variable $x_i$ of Model \ref{model1} has a value of 1 if there is a landmark at $v_i$, and 0 otherwise. Let $X$ be the set of all vertices for which $x_i=1$ in a feasible solution of Model \ref{model1}. Note that $\dist_r(a, b) = \min(\dist(a, b), r+1)$. Thus, if the constraint $\sum_{k = 1}^n \left|\min(\dist(v_k,v_i),r+1)-\min(\dist(v_k,v_j),r+1)\right|x_k > 0$ is satisfied for all pairs of vertices $v_i,v_j$, then there exists at least one vertex $v_k$ with a landmark on it such that $\dist_r(v_k, x_i) \neq \dist_r(v_k, v_j)$. Therefore, $X$ is a distance-$r$ resolving set of $G$, and since the objective function minimizes the number of vertices in the set, the optimal value of Model \ref{model1} is the minimum size of a distance-$r$ resolving set of $G$, i.e., $\dim_r(G)$.
\end{proof}

To compute $\thr(G)$, we run Model \ref{model1} a total of $n+1$ times to compute $\dim_r(G)$ with $r=0,1,\ldots,n$, and then find the minimum of $r+\dim_r(G)$ over all values of $r$. We implemented this method in Python (the code is available in \cite{diao_LP})  and ran it for various graphs; we were  able to compute the metric dimension throttling numbers of graphs on up to several hundred vertices in several minutes. Some computational results are shown in Section \ref{sec:standard}. Analogous algorithms (with the appropriate modification to Model \ref{model1}) can be defined for edge metric dimension, mixed metric dimension, and any subset-variant of metric dimension.

\section{Throttling for subset-variants of metric dimension}\label{sec:general}

In this section, we obtain several general results that hold for all subset-variants of metric dimension throttling. We begin with a general lower bound on throttling numbers, with respect to the size of the set of subsets which must be resolved. We will show in Section~\ref{stdimextr} that this bound is sharp for standard metric dimension.

\begin{thm}\label{thm:minthx}
    For every subset-variant $\xdim$ of metric dimension and every graph $G$, if $N = |T_{\xdim}(G)|$, then  \[\thrx(G) = \Omega \left(\frac{\log{N}}{\log{\log{N}}}\right).\] 
\end{thm}

\begin{proof}
    Suppose that $\thrx(G)$ can be achieved with $b$ landmarks and radius at most $a$.
Let $x = \max(a, b)$ and observe the radius is at most $x$ and there are at most $x$ landmarks. 

Each $x$-truncated distance vector has at most $x$ coordinates and each coordinate has at most $x+2$ possibilities ($0,1, \dots, x+1$). Since each element of $T_{\xdim}$ has a unique vector, it is necessary that $(x+2)^x \ge N$, which implies that $(x+2)^{x+2} > N$. If we let $y^y=N$, then $y \log{y} = \log{N}$ and $y<\log N$, which implies $y = \frac{\log N}{\log y} > \frac{\log{N}}{\log{\log{N}}}$. Thus, $\thrx(G) = \Omega \left(\frac{\log{N}}{\log{\log{N}}}\right)$.
\end{proof}

We next give general upper and lower bounds on throttling numbers for subset-variants of metric dimension, with respect to the order of the graph. 

\begin{thm}\label{thm:min_max_gen}
    For all graphs $G$ of order $n$ and all subset-variants $\xdim$ of metric dimension, $\xdim(G) \le \thrx(G) \le n$. 
\end{thm}

\begin{proof}
    The lower bound follows by definition, since it is impossible to $\xdim$-resolve $G$ with fewer than $\xdim(G)$ landmarks. The upper bound follows from setting $r=0$ and placing landmarks on all of the vertices. Indeed, consider any two distinct subsets $S, T \subseteq V(G)$. Without loss of generality, we may assume that there is some element $v \in T$ which is not in $S$. Then, $v$ distinguishes $T$ from $S$, since $\dist(v, T) = 0$ and $\dist(v, S) > 0$.
    Hence $\thrx(G)\le0+\xdim_0(G)\le n$.
\end{proof}

Both the upper and lower bounds in Theorem~\ref{thm:min_max_gen} are sharp for mixed metric dimension: when $G = K_n$, the only vertex that distinguishes an edge $e=uv$ and vertex $u$ is $v$, so $\mdim(G) = \thrm(G) = n$.


Next, we prove a slightly stronger upper bound for subset-variants $\xdim$ in which the set of subsets to be resolved all have the same size; this includes standard and edge metric dimension.

\begin{thm}\label{thm:min_max}
    For all graphs $G$ of order $n$ and all subset-variants $\xdim$ of metric dimension where the elements of $T_{\xdim}(G)$ all have the same size, $\xdim(G)\le\thrx(G) \le n-1$. 
\end{thm}

\begin{proof}
    The lower bound follows by definition, since it is impossible to $\xdim$-resolve $G$ with fewer than $\xdim(G)$ landmarks. The upper bound follows from placing landmarks on all but one of the vertices. Indeed, consider any two distinct subsets $S, T \subseteq V(G)$ of size $k$. There is some element $v \in T$ which is not in $S$. We consider two cases. For the first case, suppose that $v$ has a landmark. In this case, $v$ distinguishes $T$ from $S$, since $\dist(v, T) = 0$ and $\dist(v, S) > 0$. For the second case, suppose that $v$ has no landmark. Then, the $\xdim$-vector of $S$ has $|S| = k$ zeroes, but the $\xdim$-vector of $T$ has $|T|-1 = k-1$ zeroes, so the landmarks distinguish $S$ and $T$ in this case as well. 
    Hence $\thrx(G)\le0+\xdim_0(G)\le n-1$.
\end{proof}

Both the upper and lower bounds in Theorem~\ref{thm:min_max} are sharp for standard and edge metric dimension: when $G = K_n$, the only vertices distinguishing two distinct vertices $u,v$ are $u$ and $v$, and the only vertices distinguishing edges $uw$ and $vw$ are $u$ and $v$. Thus, $\dim(G) =\thr(G)= n-1$ and $\edim(G) =\thre(G)= n-1$ .


Using Theorem~\ref{thm:min_max_gen}, we obtain a general corollary about throttling with respect to diameter.

\begin{cor}
    For all subset-variants $\xdim$ of metric dimension and all graphs $G$ of diameter $D$, we have $\thrx(G) \in [\xdim(G), \xdim(G)+D-1]$. In particular, if $D = O(\xdim(G))$, then we have $\thrx(G) = \Theta(\xdim(G))$.
\end{cor}

Clearly, the second part of the last corollary applies to all graphs $G$ of diameter $c$, for any constant $c$. However, it would not apply in the case that $G = P_n$ and $\xdim \in \{\dim, \edim, \mdim\}$, since $P_n$ has diameter $n-1$, $\dim(P_n) = \edim(P_n) = 1$, and $\mdim(P_n) = 2$ for all $n > 1$.

In the following theorem, we prove a general lower bound in terms of diameter for subset-variants $\xdim$ of metric dimension for which $V \subseteq T_{\xdim}(G)$ or $E \subseteq T_{\xdim}(G)$. Note that this applies to standard metric dimension, edge metric dimension, and mixed metric dimension.

\begin{thm}\label{diamd}
For all graphs $G$ of diameter $d$ and all subset-variants $\xdim$ of metric dimension for which $V \subseteq T_{\xdim}(G)$ or $E \subseteq T_{\xdim}(G)$, we have $\thrx(G) = \Omega(\sqrt{d})$. 
\end{thm}

\begin{proof}
Suppose that $G$ has diameter $d$. Let $P=v_1, \dots, v_{d+1}$ be a minimal path between two vertices $v_1$ and $v_{d+1}$ with distance $d$ in $G$. Suppose that $S$ is the set of vertices at which we place $k$ landmarks in $G$, and that the robot uses an $r$-sensor to detect landmarks. 

We consider two cases. First, suppose that $V \subseteq T_{\xdim}(G)$. At most one vertex among $v_1, \dots, v_{d+1}$ can be more than $r$ away from every landmark. By pigeonhole principle, there exists $u\in S$ such that at least $d/k$ vertices on $P$ are within distance $r$ of $u$. Let $v_i$ and $v_j$ be such vertices closest to either ends of $P$. Applying triangular inequality on vertices $v_i,v_j$, and $u$ we have $2r+1 \geq d/k$. Thus $k + r \geq k + d/2k-1/2=\Omega(\sqrt{d})$ by the arithmetic-geometric mean inequality. Hence $\thrx(G) = \Omega(\sqrt{d})$.

Now, suppose that $E \subseteq T_{\xdim}(G)$.
At most one edge of $P$ can be more than $r$ away from every landmark. By pigeonhole principle, there exists $u\in S$ such that at least $(d-1)/k$ edges of $P$ are within distance $r$ of $u$. Let $v_i$ and $v_j$ be endpoints of such edges closest to either ends of $P$. Applying triangular inequality on vertices $v_i,v_j$, and $u$ we have $2(r+1) \geq (d-1)/k$. Thus $k + r \geq k + (d-2k-1)/2k=\Omega(\sqrt{d})$ by the arithmetic-geometric mean inequality. Hence $\thrx(G) = \Omega(\sqrt{d})$.
\end{proof}

In Section~\ref{cyc_path}, we show that for standard metric dimension, Theorem \ref{diamd} is sharp up to a constant factor for paths, cycles, and spiders with a constant number of legs.

Next, we obtain some general characterizations of graphs with low throttling numbers for subset-variants of metric dimension.

\begin{thm}\label{thmjoined}
    Let $G$ be a graph and $\xdim$ be a subset-variant of metric dimension. Then, 
    \begin{itemize}
    \item $\thrx(G) = 0$ if and only if $|T_{\xdim}(G)| \le 1$;
    \item $\thrx(G) = 1$ if and only if $|T_{\xdim}(G)| = 2$.
    \end{itemize}
\end{thm}
\begin{proof}
    The only way $\thrx(G) = 0$ is if no landmarks are required to resolve the subsets in $T_{\xdim}(G)$. Thus, there cannot be two distinct subsets in $T_{\xdim}(G)$, so $|T_{\xdim}(G)| \le 1$.     
    Likewise, the only way $\thrx(G) = 1$ is if $G$ can be distance-$0$ resolved with a single landmark and $|T_{\xdim}(G)| > 1$. If $G$ can be distance-$0$ resolved with a single landmark, then there are only two possible distance vectors, so  $|T_{\xdim}(G)| = 2$. For the other direction, if $|T_{\xdim}(G)|=2$ then we can place a landmark on a vertex that belongs to only one subset in $T_{\xdim}(G)$.
\end{proof}




Theorem~\ref{thmjoined} implies the following characterizations. 
\begin{cor}
For any graph $G$,
\begin{enumerate}
\item
$\thr(G) = 0$ if and only if $G$ has at most one vertex, and $\thr(G) = 1$ if and only if $G$ has two vertices.
\item
$\thre(G) = 0$ if and only if $G$ has at most one edge, and $\thre(G) = 1$ if and only if $G$ has two edges.
\item
$\thrm(G) = 0$ if and only if $G$ has at most one vertex, and $\thre(G) = 1$ if and only if $G$ has two vertices and no edges.
\end{enumerate}
\end{cor}
Finally, we show that all subset-variants of metric dimension are subtree-monotone.

\begin{prop}\label{thm:subtree}
    If $\xdim$ is a subset-variant of metric dimension and $G$ is a tree with subtree $G'$, then $\thrx(G') \le \thrx(G)$.
\end{prop}

\begin{proof}
    Let $S$ be a distance-$r$ $\xdim$-resolving set for $G$ of size $\xdim_r(G)$, where $r+\xdim_r(G) = \thrx(G)$. Let $S'$ be the set of vertices obtained from $S$ by replacing every vertex $v \in S$ that is not in $G'$ with the vertex in $G'$ that is closest to $v$. The resulting vertex set $S'$ must be a distance-$r$ $\xdim$-resolving set for $G'$, and $|S'| \le |S|$, so we have $\thrx(G') \le \thrx(G)$.
\end{proof}

\section{Standard metric dimension throttling}\label{sec:standard}

In this section, we focus on throttling for standard metric dimension. In particular, we obtain a number of extremal bounds and characterizations, as well as exact or close-to-exact values of the throttling numbers of special families of graphs.  

\subsection{Extremal bounds}
\label{stdimextr}

We begin by restating Theorem~\ref{thm:minthx} in terms of standard metric dimension and showing that it is sharp up to a constant factor in this case.

\begin{thm}\label{minth}
For all graphs $G$ of order $n$,  $\thr(G) = \Omega\left(\frac{\log{n}}{\log{\log{n}}}\right)$.
\end{thm}

\begin{thm}
There exist graphs $H_{n}$ of order $n$ with $\thr(H_{n}) = O\left(\frac{\log{n}}{\log{\log{n}}}\right)$.
\end{thm}

\begin{proof}
Hernando et al. \cite{hernando} proved for all positive integers $D$ and $k$ that there exist connected graphs of diameter $D$ and metric dimension $k$ with order 
\[\left(\left\lfloor\frac{2D}{3} \right\rfloor+1\right)^{k} + k \sum_{i = 1}^{\lceil D/3 \rceil} (2i-1)^{k-1} > \left(\frac{2D}{3}\right)^{k}.\] Let $D = \frac{3}{2} \left(\frac{\log{n}}{\log{\log{n}}}\right)$ and $k = 2 \left(\frac{\log{n}}{\log{\log{n}}}\right)$. Then 

\bea
\log{\left(\frac{2D}{3}\right)^{k}} &=& k \log{\left(\frac{2D}{3}\right)}\\
 &=& \frac{2\log{n}}{\log{\log{n}}} \left(\log{\log{n}}-\log{\log{\log{n}}}\right) \\
 &=& 2\left(1-o(1)\right)\log{n}.
\eea

 Thus there exist connected graphs of order at least $n^{2-o(1)}$ with diameter $\frac{3}{2} \left(\frac{\log{n}}{\log{\log{n}}}\right)$ and metric dimension $2 \left(\frac{\log{n}}{\log{\log{n}}}\right)$. By the definition of metric dimension throttling, this implies the result.
\end{proof}

\begin{cor}
    The minimum possible metric dimension throttling number of any graph of order $n$ is $\Theta\left(\frac{\log{n}}{\log{\log{n}}}\right)$.
\end{cor}

Similarly, we determine the minimum possible metric dimension throttling number of any graph with $m$ edges.

\begin{thm}
    The minimum possible metric dimension throttling number of any graph with $m$ edges is $\Theta\left(\frac{\log{m}}{\log{\log{m}}}\right)$.
\end{thm}

\begin{proof}
    If $G$ has $m$ edges, then $G$ has order $\Omega(\sqrt{m})$, so $\thr(G) = \Omega\left(\frac{\log{m}}{\log{\log{m}}}\right)$.
    For the upper bound, note that the construction from Hernando et al. \cite{hernando} is connected, so there exist connected graphs with at least $n^{2-o(1)}$ edges, diameter $\frac{3}{2} \left(\frac{\log{n}}{\log{\log{n}}}\right)$, and metric dimension $2 \left(\frac{\log{n}}{\log{\log{n}}}\right)$.
\end{proof}





Theorem~\ref{thmjoined} characterized the graphs with a small metric dimension throttling number. We now characterize the graphs with a large metric dimension throttling number.
Given three distinct vertices $x,y,z$, we say that vertices $x$ and $y$ form a \textit{good pair} if $N(x)-\{y,z\}=N(y)-\{x,z\}$. 
\begin{lem}\label{lem:n-1}
Let $G$ be a graph of order $n$. If $\thr(G)=n-1$, then for any three distinct vertices $x$, $y$, and $z$ of $G$ there exist two distinct vertices $x$ and $y$ such that $N(x)-\{y,z\}=N(y)-\{x,z\}$.
\end{lem}
\begin{proof}
   Suppose that for some three distinct vertices $x$, $y$, and $z$ we have $N(x)-\{y,z\}\ne N(y)-\{x,z\}$, $N(z)-\{x,y\}\ne N(y)-\{x,z\}$, and $N(z)-\{x,y\}\ne N(x)-\{z,y\}$. Then $V(G)-\{x,y,z\}$ is a distance-$1$ resolving set of $G$ and $\thr(G)\le1+n-3=n-2$.
\end{proof}

\begin{thm}\label{thm:th-n-1}
Let $G$ be a graph of order $n$ with $n\ge3$, then $\thr(G)=n-1$ if and only if among every three distinct vertices $x$, $y$, and $z$ of $G$ there exist two distinct vertices $x$ and $y$ such that $N(x)-\{y,z\}=N(y)-\{x,z\}$.
\end{thm}

\begin{proof}
The forward direction follows from Lemma~\ref{lem:n-1}. For the backward direction, it is clear that $G$ has no distance-$1$ resolving set of size $n-3$. For $n\ge4$ it suffices to show that among every four vertices $x,y,z,w\in V(G)$ there exist two vertices $x$ and $y$ that cannot be resolved by $V(G)-\{x,y,z,w\}$. If $x,y,z$ are pairwise adjacent or pairwise non-adjacent where $x$ and $y$ form a good pair in $x,y$, and $z$, then $x$ and $y$ cannot be resolved by $V(G)-\{x,y\}$. Suppose that the subgraph of $G$ induced by $\{x,y,z,w\}$ is isomorphic to $C_4$, which corresponds to $xyzwx$. The good pair in $x,y$, and $z$ is $x$ and $z$, which cannot be resolved by $G(V)-\{x,z\}$. Suppose that the subgraph of $G$ induced by $\{x,y,z,w\}$ is isomorphic to $2P_2$, which corresponds to $xy$ and $zw$. Then $x$ and $y$ is the good pair in $x,y$, and $z$, and $x$ and $y$ cannot be resolved by $G(V)-\{x,y\}$. The only remaining case is when the subgraph of $G$ induced by $\{x,y,z,w\}$ is isomorphic to $P_4$. Let the corresponding path be $xyzw$. Applying the necessary condition to all four subsets of $\{x,y,z,w\}$ of size three, we see that all four vertices have the same neighbors in $V(G)-\{x,y,z,w\}$. Hence, they cannot be resolved by $V(G)-\{x,y,z,w\}$.
\end{proof}

\begin{cor}\label{cor:complement}
For any graph $G$ of order $n$, $\thr(G)=n-1$ if and only if $\thr(\overline{G})=n-1$.
\end{cor}

Theorem~\ref{thm:th-n-1} says that a graph of order $n\ge3$ has throttling number $n-1$ if and only if there is a good pair in every three distinct vertices.
We say that vertices $x$ and $y$ form a \textit{near-identical pair} 
if $N(x)-\{y\}=N(y)-\{x\}$. Let $G_p$ denote a graph isomorphic to $K_p$ or $\overline{K_p}$.
\begin{thm}
Let $G$ be a graph of order $n$ with $n\ge3$. Then, $\thr(G)=n-1$ if and only if $G$ is one of the following, or its complement: 

\begin{itemize}
\item union of a star and any number of isolated vertices
\item $G_p+G_{n-p}$ for any integer $p\in[0,n]$
\item $P_4$
\end{itemize} 
\end{thm}

\begin{proof}
We prove the forward direction, and the backward direction can be easily checked. Denote by $\D$ the degree set of $G$.

    The necessary and sufficient condition in Theorem~\ref{thm:th-n-1} implies that in every three distinct vertices $x,y$, and $z$ of $G$ the two vertices $x$ and $y$ that form a good pair have degrees that differ by at most one. Moreover, if $\deg(x)>\deg(y)$, then $xz\in E(G)$ and $yz\notin E(G)$. If $\deg(x)=\deg(y)$, then $x$ and $y$ form a near-identical pair.

    First, we show that $\D$ is at most of size $3$. Suppose that $\D$ has elements $d_1<d_2<d_3<d_4$. If $d_1+2\le d_2$, then in three vertices of degrees $d_1,d_2$, and $d_4$ there do not exist distinct vertices with degrees that differ by at most one, a contradiction. Hence, we have $d_1+1=d_2$ and $d_3+1=d_4$. By selecting three vertices of degrees $d_1,d_2$, and $d_4$, respectively, we see that every vertex of degree $d_1$ is not adjacent to any vertex of degree $d_4$. By selecting three vertices of degrees $d_1,d_3$, and $d_4$, respectively, we see that every vertex of degree $d_1$ is adjacent to every vertex of degree $d_4$. We have a contradiction, and thus $\D$ has size at most $3$.

Suppose that $\D=\{d_1,d_2,d_3\}$ where $d_1+1=d_2\le d_3-2$, then $G$ has only a vertex of degree $d_3$. 
Define $V'$ as the set of vertices of degree $d_1$ or $d_2$.
$V'$ is an independent set or a clique. 
Therefore, $G$ is the union of a star and isolated vertices or its complement. The same conclusion holds if we have $d_1+2\le d_2=d_3-1$.

Suppose that $\D=\{d-1,d,d+1\}$. Denote by $A,B$, and $C$ the sets of vertices of degree $d-1,d$, and $d+1$, respectively. By considering two vertices from $A$ and one vertex from $C$ or one vertex from $A$ and two vertices from $C$, we see that every pair of vertices in $A$ and every pair of vertices in $C$ are near-identical pairs, $x$ and $y$ are adjacent for every $x\in A,y\in C$ or $x$ and $y$ are not adjacent for every $x\in A,y\in C$, every vertex in $B$ is adjacent to all vertices in $A$ or not adjacent to any vertices in $A$,  every vertex in $B$ is adjacent to all vertices in $C$ or not adjacent to any vertices in $C$, and each of $A$ and $C$ is a clique or an independent set.
There are no vertices $x\in A,z\in B,y\in C$ such that $xz\in E(G)$ and $yz\notin E(G)$, or else there is no good pair in $x,y$, and $z$.
$B$ does not have distinct vertices $z_1$ and $z_2$ such that $z_1x,z_1y,z_2y\in E(G)$ and $z_2x\notin E(G)$. This is because we need $x$ to be adjacent to $y$ for $x,y,z_1$ to have a good pair, which is $x$ and $z_1$. So $z_1$ and $z_2$ are not adjacent, and there is no good pair among $x,z_1,z_2$, a contradiction. Three cases now remain: every vertex in $B$ is adjacent to every vertex in $A\cup C$, every vertex in $B$ is not adjacent to any vertex in $A\cup C$, and every vertex in $B$ is adjacent to every vertex in $C$ and not adjacent to any vertex in $A$. In the first case, for every $x\in A,z\in B, y\in C$ we have $xy\notin E(G)$ and $x$ and $z$ is a good pair in $x,y$, and $z$. Hence $z$ is adjacent to all other vertices in $B$ and $\deg(z)=n-1$, a contradiction. In the second case, every vertex in $B$ has zero degree, a contradiction. We further split the third case into two subcases. The first subcase is that every vertex in $A$ is not adjacent to any vertex in $C$. This implies that $|C|=1$, $A$ is an independent set, $d=1$, $B$ is an independent set, and $|B|=2$. Thus, $G$ is a union of $P_3$ and some isolated vertices. The second subcase is that every vertex in $A$ is adjacent to every vertex in $C$. This implies that $|A|=1$ and $B$ is a clique. Since every vertex in $B$ has degree one more than $x$, we have $|B|=2$ and $C$ is a clique. Therefore the complement $\overline{G}$ of $G$ is a union of a star and some isolated vertices.

   Suppose that $\D=\{d_1,d_1+1\}$. Denote by $A$ and $B$ the sets of vertices of degree $d_1$ and $d_1+1$, respectively. Suppose that $x\in B$, $y,z\in A$, $x$ is adjacent to $y$, and $x$ is not adjacent to $z$. So $x,z$ is the good pair in $x,y,z$ and $y$ is not adjacent to $z$. If $A$ has another vertex $y'$ adjacent to $x$, then by symmetry we also have $y'z\notin E(G)$ and $x$ and $z$ do not have identical neighbors in $V(G)-\{x,y,z\}$, a contradiction. Thus, $A$ does not have another vertex $y'$ adjacent to $x$. Set $A'=A-\{y\}$. We see that $A'$ is an independent set and that no vertex in $A'$ is adjacent to $y$. Given $\deg(y)<\deg(x)$, $B$ has another vertex $x'$ adjacent to $x$ that is not adjacent to $y$. For every $z'\in A'$, $x,z'$ is a good pair in $x,y,z'$. Hence $x'$ is adjacent to all vertices in $A'$. By symmetry $x'$ has at most a neighbor in $A$, so $A'=\{z\}$, $A=\{y,z\}$, and $G$ is isomorphic to $P_4$. Next, suppose that every vertex in $B$ is adjacent or not adjacent to all vertices in $A$. Moreover, there exist $x\in A$ and $y,z\in B$ such that $x$ is adjacent to $y$ and $x$ is not adjacent to $z$. Then $x,y$ is a good pair in $x,y,z$ and $yz\in E(G)$. We can similarly deduce that $B$ does not have another vertex that is not adjacent to $x$. Set $B'=B-\{z\}$. We see that every vertex of $B'$ is adjacent to $z$. Given $\deg(x)<\deg(z)$, $A$ has another vertex $x'$ that is adjacent to $z$ and not adjacent to $x$. So $z$ is adjacent to $x\in A$ and not adjacent to $x'\in A$, a contradiction. The remaining cases are $xy\in E(G)$ for every $x\in A,y\in B$ and $xy\notin E(G)$ for every $x\in A,y\in B$. By Corollary~\ref{cor:complement} it suffices to discuss only the second case. It could be seen that each of $A$ and $B$ is a clique or an independent set, and thus the statement follows.

   Suppose that $\D=\{d_1,d_2\}$ where $d_1\le d_2-2$. Denote by $A$ and $B$ the sets of vertices of degree $d_1$ and $d_2$, respectively. We see that every pair of vertices in $A$ and every pair of vertices in $B$ is a near-identical pair, and $G$ is $G_p+G_{n-p}$ for some integer $p\in[0,n]$ or its complement.

   Suppose that $\D=\{d\}$. If $G$ has three distinct vertices $x,y$, and $z$ where $zx,zy\in E(G)$ and $xy\notin E(G)$, then $N(x)=N(y)$. We claim that every vertex $z'\in N(y)$ is adjacent to every vertex $w\in V(G)-N(y)\cup\{x,y\}$. If $z'$ is not adjacent to $w$, then among $z',y$, and $w$ it is $z'$ and $y$ that have identical neighbors in the remaining $n-3$ vertices. This contradicts $x$ being adjacent to $z'$ but not $y$. This implies that $N(w)=N(y)$, and $\overline{G}=K_{p}+G_{n-p}$  where $n-p=|N(y)|$.

   If $\D=\{d\}$ and $G$ does not have three distinct vertices $x,y$, and $z$ where $zx,zy\in E(G)$ and $xy\notin E(G)$, then in $G$ vertex adjacency is an equivalence relation and the complement $\overline{G}$ of $G$ is complete $r$-partite for some $r\in\mathbb{N}$. By Corollary~\ref{cor:complement}, $\thr(G)=n-1$. Each partite of $\overline{G}$ has $1$ or $n/2$ vertices and $r$ is $n$ or $2$. Therefore, $\overline{G}$ is $K_n$ or $K_{n/2,n/2}$.
   \end{proof}

\subsection{Special families of graphs}\label{cyc_path}

In this section we investigate standard metric dimension throttling for special families of graphs. We begin with the minimum possible throttling number of trees.
\begin{thm}\label{thm:thr-tree}
The minimum possible metric dimension throttling number of a tree of order $n$ is $\Theta(n^{1/3})$.
\end{thm}
\begin{proof}
For the upper bound, construct a tree $G$ as follows. Start with path $P$ of length $\Theta(n^{2/3})$, and attach legs of length $\Theta(n^{1/3})$ to every vertex of $P$. Place $\Theta(n^{1/3})$ landmarks along $P$ with distance $\Theta(n^{1/3})$ between consecutive landmarks, and it is easy to see that $\thr(G)=O(n^{1/3})$. See Figure \ref{fig:411} for a visualization of this construction.

For the lower bound, suppose that $G$ is a tree of order $n$ and $k=r+\dim_r(G)$ for some integer $r\in[0,k]$. Let $S$ be a distance-$r$ resolving set for $G$ of size $\dim_r(G)$. Consider the subgraph $G'$ of $G$ formed by all the paths between the vertices in $S$. Let the vertex set $T\subseteq V(G')$ consist of those with degree at least $3$ in $G'$. We claim that $|T|\le|S|$. Starting with an empty vertex set, we construct $G'$ by adding vertices in $S$ one by one to $G'$. Each time a new vertex from $S$ is attached to $G'$ via a path, so $|T|$ increases by at most one and hence $|T|\le |S|$. In $G'$ there are $|S\cup T|-1=\Theta\left(\dim_r(G)\right)=O(k)$ maximal paths consisting of vertices of degree at most two without any vertex in $S$ as internal vertex, and each of them has length $O(r)=O(k)$ for its middle vertices to be resolved. This implies that $|V(G')|\le|S\cup T|O(k)=O(k^2)$. Every vertex of $G$ that is not in $G'$ is attached to $G'$ via a path, and every vertex $u$ of $G'$ has at most one such path attached to it or else some vertices from the multiple paths attached to $u$ cannot be resolved. Moreover, every such path has length $O(r)=O(k)$ or else its vertices farthest from $G'$ cannot be resolved. Hence $n=O\left(|V(G')|k\right)=O(k^3)$ and $k=\Omega(n^{1/3})$.
\end{proof}

\begin{figure}[h]
    \centering
    \includegraphics[width=0.35\textwidth]{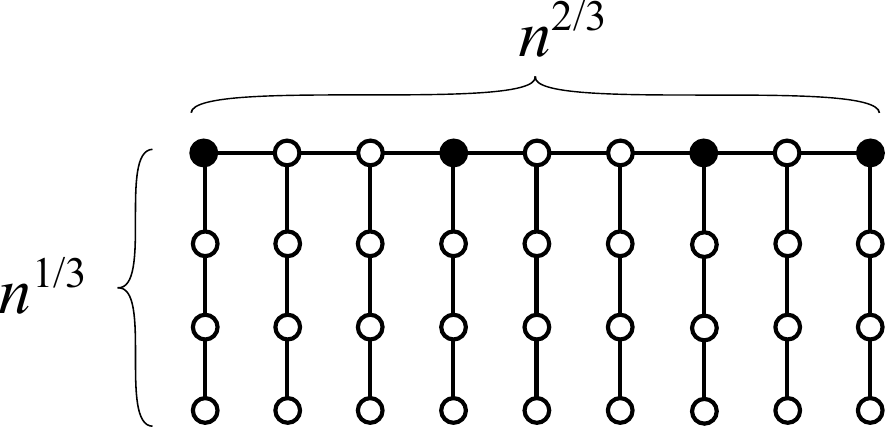}
    \caption{A $O(n^{1/3})$ metric dimension throttling configuration for a tree of order $n$.}
    \label{fig:411}
\end{figure}



Theorem~\ref{thm:min_max} implies that $\thr(K_n) = n-1$. We continue by deriving the metric dimension throttling numbers of complete bipartite graphs.

\begin{prop}
For all positive integers $s, t$, $\thr(K_{s, t}) = s+t-1$.
\end{prop}

\begin{proof}
The upper bound is achieved by placing landmarks on all but one vertex, so it suffices to prove the lower bound. From \cite{kpartite}, it is known that $\dim(K_{s, t}) \geq s+t-2$, and any resolving set of $K_{s, t}$ must contain at least every vertex but one from each part of $K_{s, t}$.

Suppose that $s+t-2$ landmarks are placed on $K_{s,t}$, since the proof would be complete if more were used. If the robot is dropped at either of the vertices with no landmarks, then it must use an $r$-sensor with $r \geq 1$ to determine its initial location. This gives a lower bound of $s+t-2+1 = s+t-1$ on $\thr(K_{s,t})$.
\end{proof}

It has been shown that the throttling numbers of paths and cycles for zero forcing, positive semidefinite zero forcing, skew zero forcing, and cop versus robber are $\Theta(\sqrt{n})$. Below we show that the throttling numbers of paths and cycles for standard metric dimension are also $\Theta(\sqrt{n})$. To this end, we employ two results on truncated metric dimension from \cite{frongillo}, stated below.

\begin{thm}\label{kdim_cycle}\cite{frongillo}
Let $n\ge 3$ and let $k$ be any positive integer. 
\begin{itemize}
\item[(a)] If $n \le 3k+3$, then $\dim_k(C_n)=2$.
\item[(b)] If $n \ge 3k+4$, then
\begin{equation*}
\dim_{k}(C_n)=\left\{
\begin{array}{ll}
\lfloor\frac{2n+3k-1}{3k+2}\rfloor & \mbox{ if } n \equiv 0,1,\ldots, k+2 \pmod{(3k+2)},\\ 
\lfloor\frac{2n+4k-1}{3k+2}\rfloor & \mbox{ if } n \equiv k+3,\ldots,  \lceil \frac{3k+5}{2}\rceil-1 \pmod{(3k+2)},\\ 
\lfloor\frac{2n+3k-1}{3k+2}\rfloor & \mbox{ if } n \equiv \lceil \frac{3k+5}{2}\rceil,\ldots, 3k+1 \pmod{(3k+2)}.
\end{array}\right.
\end{equation*}
\end{itemize}
\end{thm}

\begin{thm}\label{kdim_path}\cite{frongillo}
Let $n \ge 2$ and let $k$ be any positive integer.
\begin{itemize}
\item[(a)] If $n \le k+2$, then $\dim_k(P_n)=1$.
\item[(b)] If $k+3 \le n \le 3k+3$, then $\dim_k(P_n)=2$.
\item[(c)] If $n \ge 3k+4$, then
\begin{equation*}
\dim_k(P_n)=\left\{
\begin{array}{ll}
\lfloor\frac{2n+3k-1}{3k+2}\rfloor & \mbox{ if } n \equiv 0,1,\ldots, k+2 \pmod{(3k+2)},\\ 
\lfloor\frac{2n+4k-1}{3k+2}\rfloor & \mbox{ if } n \equiv k+3,\ldots,  \lceil \frac{3k+5}{2}\rceil-1 \pmod{(3k+2)},\\ 
\lfloor\frac{2n+3k-1}{3k+2}\rfloor & \mbox{ if }  n \equiv \lceil \frac{3k+5}{2}\rceil,\ldots, 3k+1 \pmod{(3k+2)}.
\end{array}\right.
\end{equation*}
\end{itemize}
\end{thm}

\begin{prop}\label{path} For $n \geq 3$,  $\thr(C_n) = 2 \sqrt{\frac{2}{3} n}\left(1\pm o(1)\right)$.
\end{prop}

\begin{proof}
By Theorem~\ref{kdim_cycle}, we have \[\frac{2n+3k-1}{3k+2}-1 \le \dim_k(C_n) \le \frac{2n+3k-1}{3k+2}+1\] for all $n \ge 3$ and $k \ge 1$. Therefore, \[k+\dim_k(C_n) = \frac{3k+2}{3}+\frac{2n}{3k+2} + O(1),\] so $\thr(C_n) = 2 \sqrt{\frac{2}{3} n}\left(1\pm o(1)\right)$ by the arithmetic mean - geometric mean inequality.
\end{proof}

\begin{prop} For $n > 0$, $\thr(P_n) = 2 \sqrt{\frac{2}{3} n}(1\pm o(1))$.
\end{prop}

\begin{proof}
The proof is analogous to the proof of Proposition \ref{path}, with the initial bound following from Theorem \ref{kdim_path}.
\end{proof}

Although we only have asymptotic bounds on $\thr(P_n)$ and $\thr(C_n)$, we are able to show that they are equal. First, using the algorithm described at the end of Section \ref{sec:complexity}, we determined the throttling numbers of paths and cycles up to order $50$, shown in Figure~\ref{graph1}. 

\begin{figure}[h]
    \centering
\includegraphics[width=0.75\textwidth]{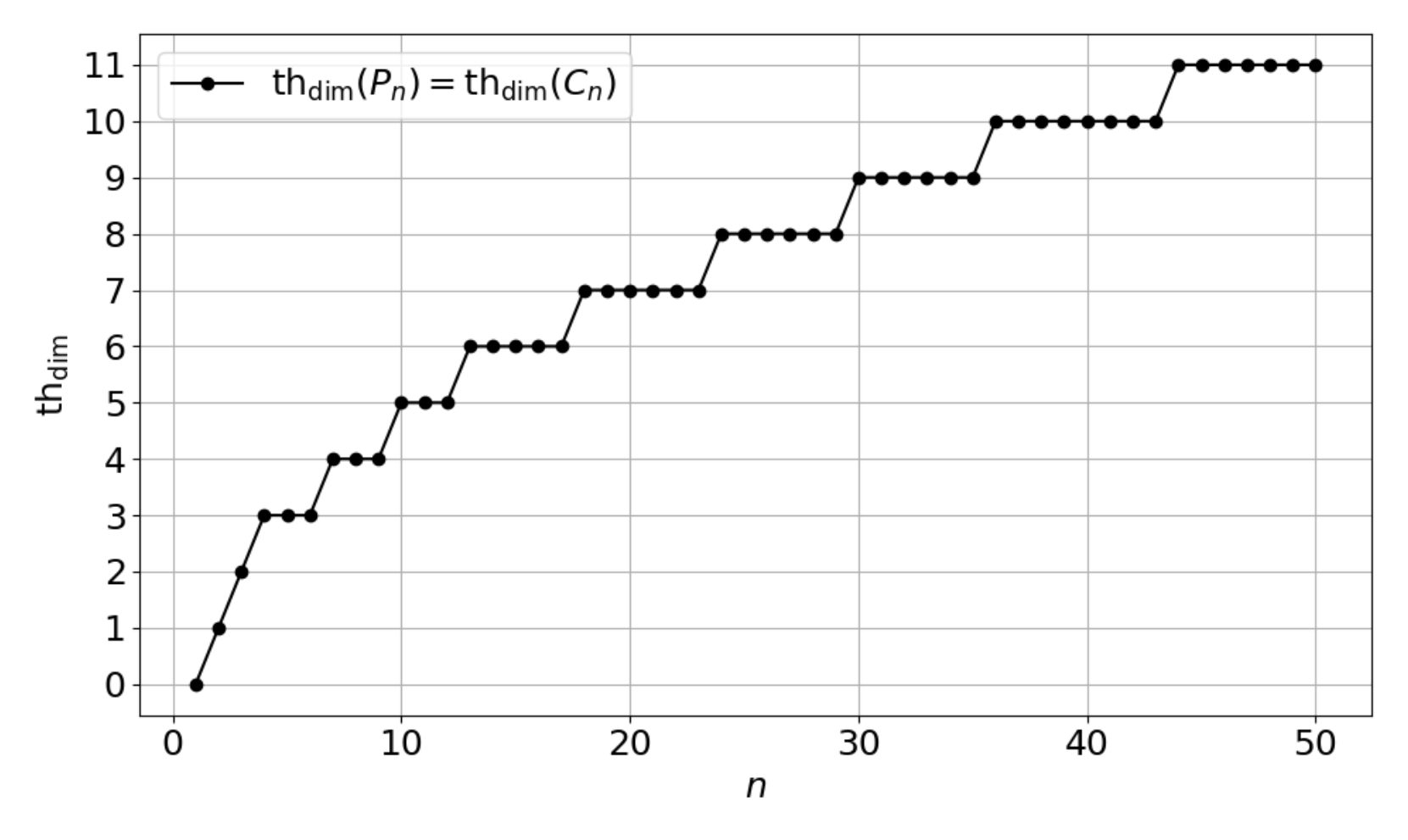}
    \caption{Metric dimension throttling numbers for $C_n$ and $P_n$, $1\leq n\leq 50$.}
    \label{graph1}
\end{figure}

    

\begin{prop}
For all $n$, $\thr(C_n) = \thr(P_n)$. 
\end{prop}

\begin{proof}
The result follows from Figure~\ref{graph1} for $n \le 50$. Now, suppose that $n > 50$. By Theorem~\ref{kdim_cycle}, we have $\dim_1(C_n) < \frac{n}{2}-1$. Similarly, by Theorem~\ref{kdim_path}, we have $\dim_1(P_n) < \frac{n}{2}-1$. Thus, $\thr(C_n) < \frac{n}{2}$ and $\thr(P_n) < \frac{n}{2}$. Note that $\dim_k(P_n) = \dim_k(C_n)$ for all $k$ except when $n \le k+2$. However, when $n \le k+2$, we have $k+\dim_k(C_n) \ge n-2$ and $k+\dim_k(P_n) \ge n-2$. 
Thus, when we throttle $C_n$ and $P_n$ for $n > 50$, we would not choose $k$ with $n \le k+2$, since we want to minimize $k+\dim_k(C_n)$ and $k+\dim_k(P_n)$, and we already have $\thr(C_n) < \frac{n}{2}$ and $\thr(P_n) < \frac{n}{2}$ by choosing $k = 1$. 
\end{proof}

Let $T_{p, \ell}$ denote the balanced spider with $p$ legs of length $\ell$. The skew zero forcing throttling number of balanced spiders was studied in \cite{skewth}. The same method from \cite{skewth} can be used to show that $\thr(T_{p, \ell}) = \Theta(\max(p, \sqrt{p \ell}))$. The case of unbalanced spiders requires a different approach. Below we characterize the standard metric dimension throttling number of unbalanced spiders with a constant number of legs.

\begin{prop}
If $S$ is any spider of order $n$ with $O(1)$ legs, then $\thr(S) = \Theta(\sqrt{n})$.
\end{prop}

\begin{proof}
We prove the lower bound first. Let $p = O(1)$ be the number of legs on $S$, $n$ be the order, and $k$ be the total number of landmarks. We know that at least one of the legs must have length at least $(n-1)/(p)$. 

By the pigeonhole principle, one of the legs with length at least $(n-1)/p$ must have at least $\frac{\frac{n-1}{p}-k}{k+1}$ consecutive vertices without a landmark. Thus $r \geq \frac{\frac{n-1}{p}-k}{2(k+1)}-2$ if the robot is using an $r$-sensor, or else at least two vertices would both be out of range of all landmarks. By the arithmetic-geometric mean inequality, we obtain $\thr(S) = \Omega(\sqrt{n})$.

To prove the upper bounds, place landmarks on the center and legs at intervals of $\ceil{\sqrt{n}}$, with at most one shorter interval formed by placing landmarks on the ends of legs. If $a_{i}$ denotes the length of the $i^{\text{th}}$ leg, then the number of landmarks used will be at most $1+\sum_{i=1}^{p}\ceil{\frac{a_{i}}{\ceil{\sqrt{n}}}} \leq (\sum_{i=1}^{p}\frac{a_{i}}{\ceil{\sqrt{n}}})+p+1 = \frac{n-1}{\ceil{\sqrt{n}}} + p+1 \leq \sqrt{n}+p+1$. 

The robot uses an $r$-sensor with $r = \ceil{\sqrt{n}}$, so it will detect at least two landmarks from which it can determine its initial location. This gives $\thr(S) = O(\sqrt{n})$. 
\end{proof}

We next study throttling for circulant graphs. Circulant graphs are of interest for metric dimension throttling since they are highly symmetric and thus it can be challenging to uniquely identify vertices. We focus on a special family of circulants $\Circ {n} { S }$, where $S\subset\mathbb{N}$ contains $1$ and a parameter $\ell$.

\begin{thm}
Let $n$ and $\ell$ be positive integers and $S \subset \left\{1, \dots, \ell\right\}$ be a subset that contains both $1$ and $\ell$. Then, $\thr(\Circ {n} { S }) = \Theta(\sqrt{n})$, where the constants depend on $\ell$.
\end{thm}

\begin{proof}
Denote by $C=v_1,v_2,\dots,v_n,v_1$ a spanning cycle of $\Circ{n}{S}$ where the vertex indices of endpoints of every edge on $C$ differ by one modulo $n$.
We begin by describing a construction which shows that the upper bound is true. Place $\ell \lceil\frac{\sqrt{n}}{\ell}\rceil$ landmarks on $G$ in $\lceil\frac{\sqrt{n}}{\ell}\rceil$ groups of $\ell$ consecutive vertices on $C$, distributed as evenly as possible. See Figure~\ref{fig:421} for an example of this construction when $n = 17$ and $\ell = 2$.

Between every pair of consecutive groups of landmarks, let the consecutive vertices without landmarks be called a sector. We suppose that $n$ is sufficiently large that there are at least four sectors. If the robot uses an $r$-sensor for $r = \lceil \sqrt{n}+\ell\rceil$, then the robot can see the distance to the two full groups of adjacent landmarks bordering its sector if the robot is not already on a landmark. 

Suppose that $u$ and $v$ are distinct vertices on $G$ with no landmarks. If $u$ and $v$ are in different sectors, then clearly $u$ and $v$ can be distinguished by the closest landmark to $u$ in the group of landmarks that borders $u$'s sector and does not border $v$'s sector. If $u$ and $v$ are in the same sector, then they both can see the distance to the two full groups of adjacent landmarks bordering their sector. Without loss of generality, suppose that $u$ has $v$ to the left and a group $X$ of landmarks to the right. Starting at $u$, we hop $\ell$ vertices at a time to the right along $C$ until we land on a vertex $q$ in $X$. Note that $d(u, q) < d(v, q)$, since the set $S$ has maximum element $\ell$ and there are at least three sectors. Thus $u$ and $v$ can be distinguished, so the construction gives $\thr(\Circ {n} { S }) = O(\sqrt{n})$.

For the lower bound, let $k$ landmarks be distributed arbitrarily. By the pigeonhole principle, some pair of landmarks $L_1$ and $L_2$ have distance on $C$ at least $n/k$ and no other landmark lies on the shortest path between $L_1$ and $L_2$. Then the robot must use an $r$-sensor with $r \ge n/2k\ell\pm O(1)$, or else the centermost vertices on the shortest path between $L_1$ and $L_2$ on $C$ would be indistinguishable. By the arithmetic-geometric mean inequality, we have $\thr\left(\Circ {n} { S }\right) = \Omega\left(k+n/2\ell k\right) = \Omega\left(\sqrt{n/\ell}\right)$.
\end{proof}
\begin{figure}[h]
    \centering
    \includegraphics[width=0.4\textwidth]{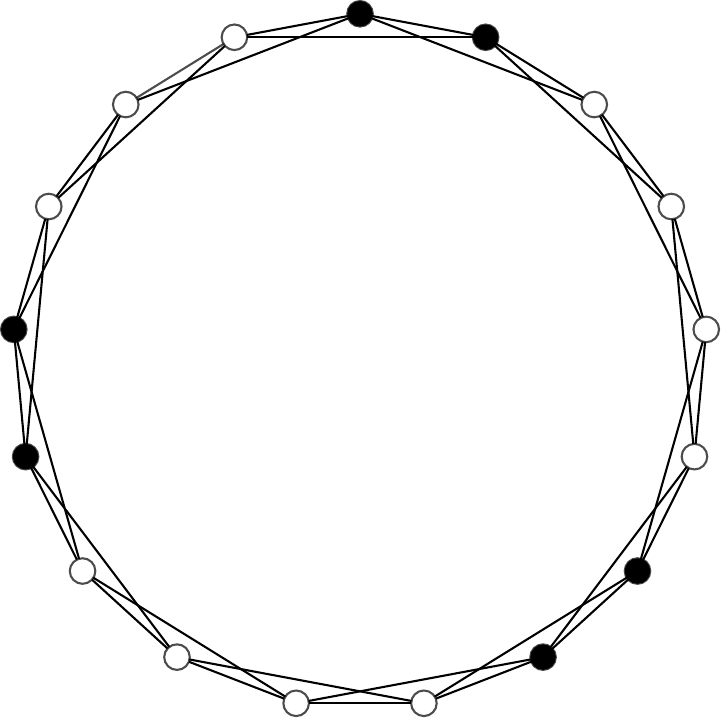}
    \caption{A metric dimension throttling configuration for $\Circ {n} { S }$ with $n=17$, $\ell=2$, $S=\{1,2\}$, and we place $6$ landmarks in $3$ groups of $2$ adjacent landmarks.}
    \label{fig:421}
\end{figure}

We conclude this section by investigating throttling for $d$-dimensional grids. For the results below, we associate every vertex $v=(v_1,v_2,\dots,v_d)$ of the graph $G=X_1\times X_2\times\dots\times X_d$ where $X_i$ is either $C_{n_i}$ or $P_{n_i}$ with a $d$-dimensional vector $(v_1,\dots,v_d)$ representing the projections $v_1,v_2,\dots,v_d$ of $v$ onto $X_1,X_2,\dots,X_d$, respectively.
\begin{lem}
\label{lem:upperbound}
    Let $G=X_1\times X_2\times\dots\times X_d$ where $X_i$ is either $C_{n_i}$ or $P_{n_i}$ for every $1\le i\le d$. Then $G$ has a distance-$r$ resolving set of size at most $\prod_{i=1}^d(2+\frac{n_i}{r})$.
\end{lem}
\begin{proof}
For every $1\le i\le d$, select a distance-$r$ resolving set $L_i\subseteq V(X_i)$ of size at most $2+\frac{n_i}{r}$, and define a vertex set $L\subset V(G)$ consisting of every vertex $v=(v_1,\dots,v_d)$ where for every $1\le i\le d$ we have $v_i\in L_i$. We prove inductively that $L$ is a distance-$r$ resolving of $G$. The statement holds for $d=1$. Given a pair of vertices $u=(u_1,\dots,u_d),v=(v_1,\dots,v_d)\in V(G)$, if $u_d=v_d$ then by inductive hypothesis they can be resolved by the projection of $L$ onto the first $d-1$ coordinates. If there does not exist a landmark $\ell\in L$ such that $\dist(u,\ell)=\dist(v,\ell)$ then we are done. Otherwise for every $\ell=(\ell_1,\dots,\ell_d)\in L$ we have
\begin{equation*}
\begin{split}
\dist(u,\ell)&=\dist\left((u_d),(\ell_d)\right)+\dist\left((u_1,\dots,u_{d-1}),(\ell_1,\dots,\ell_{d-1})\right)\\
\dist(v,\ell)&=\dist\left((v_d),(\ell_d)\right)+\dist\left((v_1,\dots,v_{d-1}),(\ell_1,\dots,\ell_{d-1})\right)\\
\end{split}
\end{equation*}
We can select $\ell$ such that without loss of generality $\dist\left((u_d),(\ell_d)\right)<\dist\left((v_d),(\ell_d)\right)$, so $$
\dist\left((u_1,\dots,u_{d-1}),(\ell_1,\dots,\ell_{d-1})\right)>\dist\left((v_1,\dots,v_{d-1}),(\ell_1,\dots,\ell_{d-1})\right).
$$
On $X_d$ there exists another landmark $\ell'_d$ such that $\dist\left((u_d),(\ell'_d)\right)>\dist\left((v_d),(\ell'_d)\right)$. Hence let
$\ell'=(\ell_1,\ell_2,\dots,\ell_{d-1},\ell'_d)\in L$, and we have $\dist(u,\ell')>\dist(v,\ell')$.
\end{proof}

\begin{lem}
\label{lem:lowerbound}
    Let $d$ be a fixed positive integer and $G=X_1\times X_2\times\dots\times X_d$ be a graph of order $n$ where $X_i$ is either $C_{n_i}$ or $P_{n_i}$ for every $1\le i\le d$. Then,
    $$
    \thr(G)=\Omega\left(n^{1/(d+1)}\right).
    $$
\end{lem}
\begin{proof}
   Suppose that $r+\dim_r(G)=\thr(G)$ for some nonnegative integer $r$. Let $L$ be a minimum distance-$r$ resolving set of $G$. At most one vertex of $G$ is at distance at least $r+1$ to every vertex in $L$. Every vertex in $L$ is at distance no more than $r$ to at most $\binom{r+d}{d}2^d$ vertices, where $\binom{r+d}{d}$ is the number of ways to select a nonnegative integer at most $r$ and write it as the sum of $d$ ordered nonnegative integers. Thus we have
   $$
   \dim_r(G)\binom{r+d}{d}2^d\ge n-1.
   $$
   Hence
   $$
   \dim_r(G)(r+d)^d\ge\frac{d!}{2^d}(n-1).
   $$
   Therefore
\begin{equation*}
\begin{split}
   \dim_r(G)+r+d&=\dim_r(G)+d\cdot\frac{r+d}{d}\\
   &\ge (d+1)\left(\dim_r(G)\left(\frac{r+d}{d}\right)^d\right)^{\frac{1}{d+1}}\\
   &\ge(d+1)\left(\frac{d!}{2^dd^d}\right)^{\frac{1}{d+1}}\left(n-1\right)^{\frac{1}{d+1}}.
\end{split}
\end{equation*}
Thus the statement follows.
\end{proof}

We combine the upper and lower bounds of Lemmas \ref{lem:upperbound} and \ref{lem:lowerbound} to obtain the following characterization of the standard metric dimension throttling number of $d$-dimensional grids. 
    
\begin{thm}
\label{thmgrid}
    Let $d$ be a fixed positive integer, $G=X_1\times X_2\times\dots\times X_d$ where $X_i$ is either $C_{n_i}$ or $P_{n_i}$ for every $1\le i\le d$, and $n_1\ge n_2\ge\dots\ge n_d$. Then,
    $$
    \thr(G)=\Theta\left(\max_{1\le i\le d}\left\{\left(n_1n_2\dots n_i\right)^{\frac{1}{i+1}}\right\}\right).
    $$
\end{thm}

\begin{proof}
    Define
    $
    i=\arg\max_{1\le j\le d}\left(\prod_{k=1}^j n_k\right)^{\frac{1}{j+1}}
    $
    and
    $
    r=\max_{1\le j\le d}\left(\prod_{k=1}^j n_k\right)^{\frac{1}{j+1}}$.

For the lower bound, consider $G'=X_1\times X_2\times\dots\times X_i$. Then 
$$
\thr(G)\ge\thr(G')=\Omega\left(\left(n_1n_2\dots n_i\right)^\frac{1}{i+1}\right),
$$
where the last equality follows from Lemma~\ref{lem:lowerbound}.

For the upper bound, we will show that $r\le n_j$ for every integer $j\in[1,i]$ and $n_k\le r$ for every integer $k\in[i,d]$. With these and Lemma~\ref{lem:upperbound}, there is a distance-$\ceil{r}$ resolving set of $G$ of size at most
    \begin{equation*}
    \begin{split}
       &(2+\frac{n_1}{r})(2+\frac{n_2}{r})\dots(2+\frac{n_i}{r})(2+\frac{n_{i+1}}{r})\dots(2+\frac{n_{d}}{r})\\
       &\le \left(\frac{3}{r}\right)^in_1n_2\dots n_i3^{d-i}.
    \end{split}
    \end{equation*}

Therefore, 
\begin{equation*}
   \begin{split}
\thr(G)&\le \ceil{r}+\left(\frac{3}{r}\right)^in_1n_2\dots n_i3^{d-i}\\
&=\left(n_1\dots n_i\right)^\frac{1}{i+1}+3^d\left(n_1\dots n_i\right)^\frac{1}{i+1}\\
&=\Theta\left(\left(n_1\dots n_i\right)^\frac{1}{i+1}\right).
   \end{split} 
\end{equation*}
    First, we show that $r\le n_i$. If $i=1$ then it is true. If $i>1$, then
    \begin{equation*}
    \begin{split}
    \left(n_1\dots n_{i-1}\right)^{\frac{1}{i}} &\le \left(n_1\dots n_{i}\right)^{\frac{1}{i+1}}\\
    n_1n_2\dots n_{i-1}&\le n_i^i\\
    n_1n_2\dots n_{i}&\le n_i^{1+i}\\
    r=\left(n_1n_2\dots n_{i}\right)^\frac{1}{i+1}&\le n_i.\\
    \end{split}
    \end{equation*}

    Next, we show that if $i<d$ then $n_{i+1}\le r$. This follows from
    \begin{equation*}
       \begin{split}
           \left(n_1n_2\dots n_{i+1}\right)^\frac{1}{i+2}&\le\left(n_1n_2\dots n_i\right)^\frac{1}{i+1}\\
           \left(n_1n_2\dots n_{i+1}\right)^{i+1}&\le\left(n_1n_2\dots n_i\right)^{i+2}\\
           n_{i+1}^{i+1}&\le n_1n_2\dots n_i\\
           n_{i+1}&\le\left(n_1\dots n_i\right)^\frac{1}{i+1}=r.
       \end{split} 
    \end{equation*}
\end{proof}

\begin{cor}
Let $d$ be a fixed positive integer and 
    $G=X_1\times X_2\times\dots\times X_d$ be a graph of order $n$ where $X_i$ is either $C_{n_i}$ or $P_{n_i}$ for every $1\le i\le d$ and $n_1\ge n_2\ge\dots\ge n_d$.
    Then $\thr(G)=O(n^{1/2})$ and $\thr(G)=\Omega(n^{1/(d+1)})$. Moreover, $\thr(G)=\Theta(n^{1/2})$ if and only if $n_1=\Theta(n)$ and $n_i=O(1)$ for $i=2,3,\dots,d$, and $\thr(G)=\Theta\left(n^{1/(d+1)}\right)$ if and only if $n_1n_2\dots n_i=O\left(n^{(i+1)/(d+1)}\right)$ for $i=1,2,\dots,d-1$.
\end{cor}

The $d$-dimensional hypercube graph $Q_d=P_2\times\ldots\times P_2$ is a special case of the grid graphs described in Theorem \ref{thmgrid}. Using the algorithm described at the end of Section \ref{sec:complexity}, we determined the exact throttling numbers of $d$-dimensional hypercubes $Q_d$ for $1 \le d \le 5$, shown in Figure \ref{graph2}.

  

\begin{figure}[h]
    \centering
\includegraphics[width=0.75\textwidth]{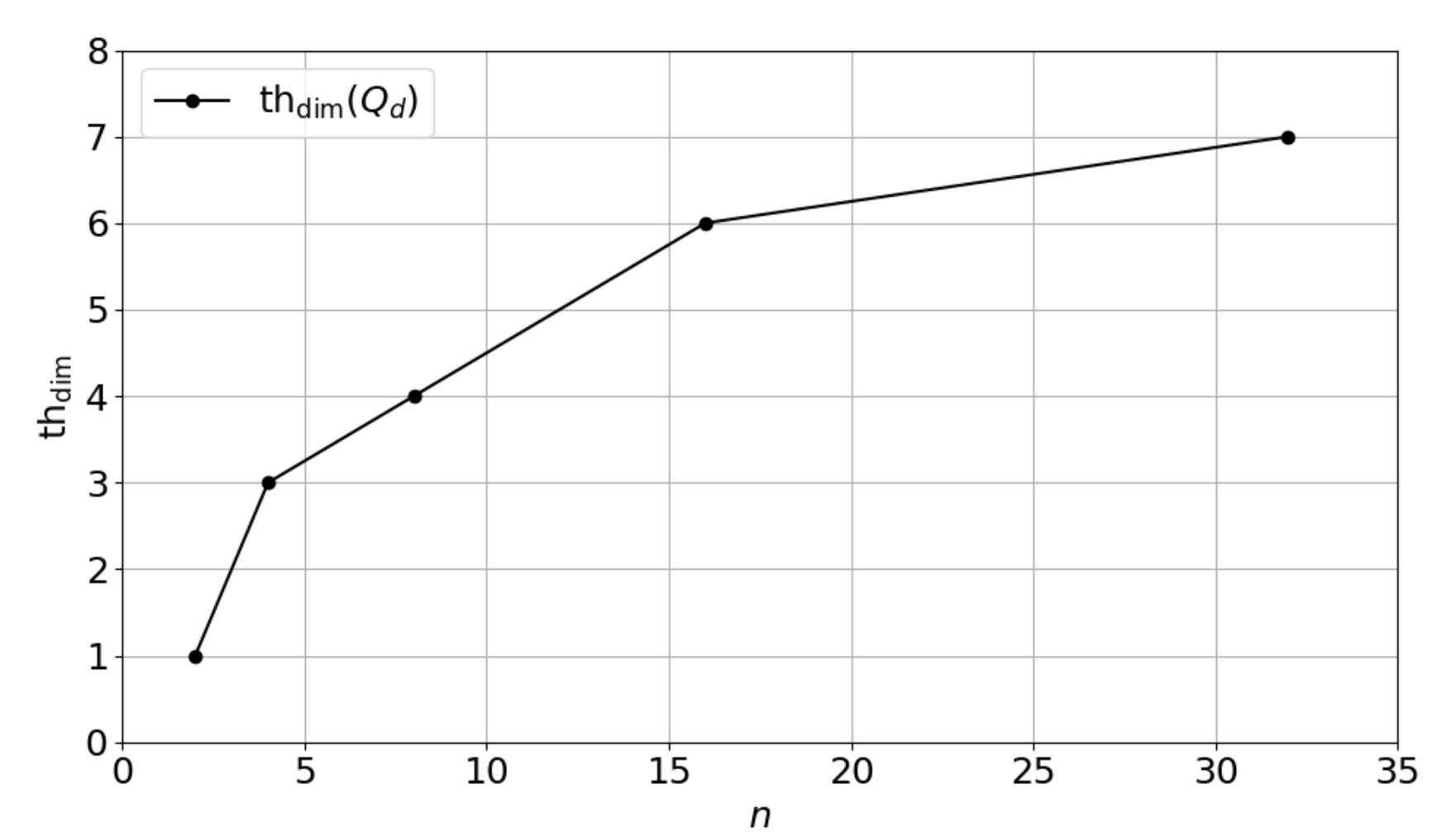}
    \caption{Metric dimension throttling numbers for $Q_d$, $1\leq d\leq 5$.}
    \label{graph2}
\end{figure}

\section{Edge metric dimension throttling}\label{sec:edge}

In this section, we give results about the edge metric dimension throttling numbers of several families of graphs.

\begin{prop}
For all positive integers $n$,
$\thre(K_n) = \begin{cases}
   0 & n=1,2\\
   n-1& n>2.
\end{cases}$
\end{prop}

\begin{proof}
This follows from Theorem~\ref{thm:min_max}, since the upper and lower bounds are matching when $\xdim = \edim$ and $G = K_n$ for all integers $n > 2$.
\end{proof}

\begin{prop}
    For all positive integers $s, t$, $\thre(K_{s,t}) = s+t-2$.
\end{prop}

\begin{proof}
    By Theorem~\ref{thm:min_max}, $\thre(K_{s,t}) \ge \edim(K_{s,t})$, and by a result in \cite{kelenc},  $\edim(K_{s,t}) = s+t-2$ . Moreover, it is easy to see that every edge $e \in E(K_{s,t})$ has distance $0$ or $1$ to every vertex $v \in V(K_{s,t})$, so $\thre(K_{s,t}) \le \edim(K_{s,t})+0 = \edim(K_{s,t})$.
\end{proof}

We next show that for cycles and paths,  the edge metric dimension throttling number is approximately the same as the standard metric dimension throttling number.

\begin{figure}[h]
    \centering
    \includegraphics[width=0.3\textwidth]{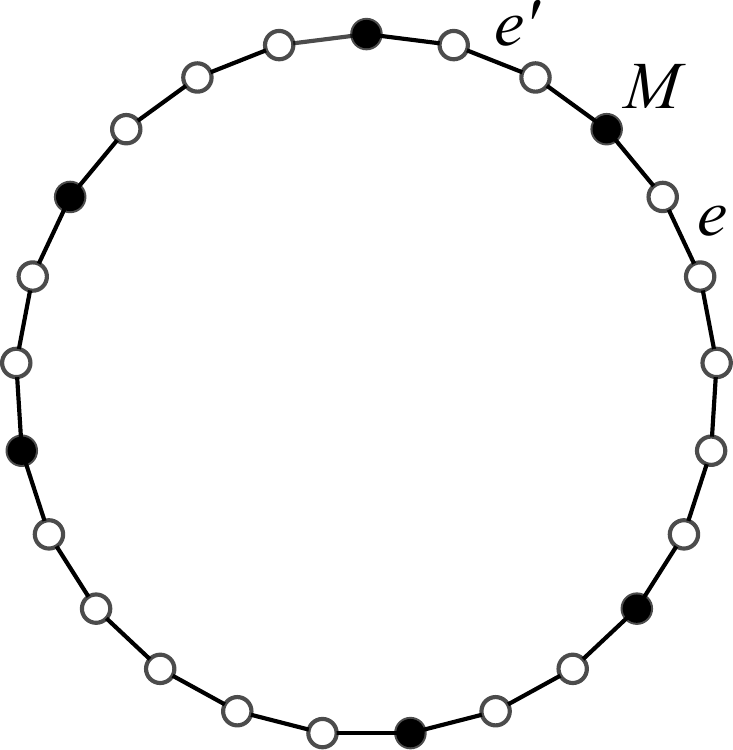}
    \caption{An optimal edge metric dimension throttling configuration for $C_n$ with $n=25$, $k=3$, $r=3$, and the landmarks are at distance $3$ and $6$ alternating, and an edge $e$ and $e'$ from the second part of the proof are marked.}
    \label{fig:58}
\end{figure}

\begin{thm}
    For all integers $n \geq 3$,  
    \begin{itemize}
\item $\thre(C_n) = 2\sqrt{\frac{2n}{3}}\paren{1\pm o(1)}$,
\item $\thre(P_n) = 2 \sqrt{\frac{2}{3} n}(1\pm o(1))$. 
    \end{itemize}
\end{thm}

\begin{proof}
First we will show $\paren{2\sqrt{\frac{2n}{3}}}(1+ o(1)) \leq \thre(C_n)$. Suppose that $r = k$, and consider consecutive landmarks $A,B,C$. Denote the two edges incident on either side of $B$ as $e_1$ and $e_2$. We claim that there are at most $k+1$ edges in the path $p_1$ between $A$ and $B$ or there are at most $k+1$ edges in the path $p_2$ between $B$ and $C$.

Seeking a contradiction, suppose that $|E(p_1)| > k+1$ and $|E(p_2)| > k+1$. Then $e_1$ and $e_2$ only register the landmark $B$ since $r = k$. This then implies that $e_1$ and $e_2$ have identical distance vectors, so we have a contradiction. As a result, at least one path on either side of any landmark must contain at most $k+1$ edges. Furthermore, observe that the number of edges between two consecutive landmarks $A$ and $B$ does not exceed $2k+3$, or else the two middle edges between $A$ and $B$ cannot be resolved. Thus on a cycle $C_n$ with $L$ landmarks, we must have $L\ge n/\paren{\paren{2k+3+k+1}/2} =n /\paren{2+3k/2}$, which implies that \[\frac{2n}{3} \leq \left(k+\frac{4}{3}\right)L \leq \left(\frac{k+L+\frac{4}{3}}{2}\right)^2.\]
Thus, $2\sqrt{2n/3}\paren{1\pm o(1)}\le\thre(C_n)$.

Next, we will show that $\thre(C_n) \leq \paren{2\sqrt{\frac{2n}{3}}}(1+ o(1))$. Let \[k = \cl{\sqrt{\frac{n}{6}}}.\] Suppose that $2k$ landmarks are placed at alternating intervals of length $\lceil\frac{n}{3k}\rceil$ and $\lceil\frac{2n}{3k}\rceil$, where some intervals may have shorter length. Furthermore, suppose that $r = \lceil\frac{n}{3k}\rceil$. See Figure~\ref{fig:58} for an example of this construction when $n = 25$.

If we select some edge $e$ in an interval of length at most $\lceil\frac{n}{3k}\rceil$, then the distance to the two closest landmarks from $e$ is at most $r$, so $e$ has a unique distance vector. If we select some edge $e$ in an interval of length greater than $\lceil\frac{n}{3k}\rceil$, then the distance to the closest landmark $M$ is at most $r$. There is only one other edge $e' \in E(C_n)$ with the same distance to $M$. However, $e'$ must lie in an interval of length at most $\lceil\frac{n}{3k}\rceil$, so $e$ has an unique distance vector. Thus, we have $\thre(C_n) \leq \paren{2\sqrt{\frac{2n}{3}}}(1+ o(1))$. Therefore, we have $\thre(C_n) = \paren{2 \sqrt{\frac{2n}{3}}}(1\pm o(1))$. 

By similar reasoning as above, $\thre(P_n) = 2 \sqrt{\frac{2}{3} n}(1\pm o(1))$.
\end{proof}

The next result on the minimum possible edge throttling number of a tree of order $n$ has a proof almost identical to the proof of Theorem~\ref{thm:thr-tree}, except that the lengths of the paths are bounded for edges to be resolved. The proof is omitted for brevity. 
\begin{thm}\label{thm:ethr-tree}
The minimum possible edge metric dimension throttling number of a tree of order $n$ is $\Theta(n^{1/3})$.
\end{thm}

Next, we find the maximum possible edge metric dimension throttling number of a tree of order $n$.

\begin{prop}
The maximum possible edge metric dimension throttling number of a tree of order $n \ge 2$ is $n-2$.
\end{prop}

\begin{proof}
    First, observe that the maximum is at least $n-2$ from the star graph. To see that $n-2$ suffices, let $T$ be a tree of order $n \ge 2$. $T$ has at least two leaves, so let $x, y$ be two leaves in $T$. If $x$ and $y$ have the same common neighbor $z$, then $V(T)-\{x,z\}$ is a distance-$0$ edge resolving set for $T$. If $x$ and $y$ have different neighbors, then $V(T)-\{x,y\}$ is a distance-$0$ edge resolving set for $T$.
\end{proof}



We conclude with an infinite family of graphs whose edge metric dimension throttling number is close to the lower bound. 
Recall that by Theorem~\ref{thm:minthx}, for all graphs $G$ with $m$ edges,  $\thre(G) = \Omega\left(\frac{\log{m}}{\log{\log{m}}}\right)$.



\begin{prop}
    There exists an infinite family of connected graphs $G$ with $m$ edges for which $\thre(G) = O\left(\log{m}\right)$.
\end{prop}

\begin{proof}
    The hypercube graph $Q_d$ has $m = d 2^{d-1}$ edges, diameter $d$, and $\edim(Q_d) = \Theta\left(\frac{d}{\log{d}}\right)$ \cite{kelenc}. 
    Thus, $\thre(Q_d) = O(d) = O(\log{m})$.
\end{proof}

\section{Mixed metric dimension throttling}\label{sec:mixed}
In this section, we give results about the mixed metric dimension throttling numbers of several families of graphs, and highlight differences between throttling for mixed metric dimension and the other subset-variants.








\begin{prop}
For all positive integers $n$,
$\thrm(K_n) = \begin{cases}
   0 & n=1\\
   n& \text{otherwise}.
\end{cases}$
\end{prop}

\begin{proof}
This follows from Theorem~\ref{thm:min_max_gen}, since the upper and lower bounds are matching when $\xdim = \mdim$ and $G = K_n$ for all integers $n > 1$.
\end{proof}


For complete bipartite graphs $K_{s,t}$, the mixed metric dimension throttling number coincides with the standard metric dimension throttling number (for $s, t \ge 2$) and differs from the edge metric dimension throttling number.

\begin{prop}
    For all positive integers $s,t$, $\thrm(K_{s,t})=\begin{cases}
    s+t&\min(s,t)\le2\\
    s+t-1&\text{otherwise}.
\end{cases}$
\end{prop}

\begin{proof}
Since mixed resolving sets must resolve both vertices and edges, while standard resolving sets must only resolve vertices, we have $\thrm(K_{s,t}) \ge \thr(K_{s,t}) = s+t-1$.
    The upper bound $s+t$ corresponds to placing a landmark on every vertex. We split the rest of the proof into three cases. For the first case, suppose that $s > 2$ and $t > 2$. Note that $\mdim(K_{s,t}) = s+t-2$, and $K_{s,t}$ has diameter $2$, so $\thrm(K_{s,t}) \le (s+t-2)+1 = s+t-1$.

    For the second case, suppose that $s \le 2$ or $t \le 2$, but $(s, t) \neq (1,1)$. First note that $\mdim(K_{s,t}) = s+t-1$. We show that there does not exist a mixed $0$-resolving set of size $s+t-1$, and therefore $\thrm(K_{s,t})\ge s+t$. Let $S=V(G)-\{v\}$. Consider any edge $e = \{u,v\}$. For any vertex $x\notin\{u,v\}$, we have $\dist(u, x) > 0$ and $\dist(e, x) > 0$, so $x$ does not $0$-resolve $e$ and $u$. Moreover, $\dist(u, u) = \dist(e, u) = 0$, so $u$ does not resolve $e$ and $u$. Thus, $S$ is not a mixed $0$-resolving set of $K_{s,t}$.

    For the third case, suppose that $s = t = 1$. In this case, $\thrm(K_{s,t}) = \thrm(K_2) = 2 = s+t$.
\end{proof}

As with standard and edge metric dimension throttling, we obtain $\Theta(\sqrt{n})$ bounds for mixed metric dimension throttling on paths and cycles. Note that the leading coefficient is $2$ for mixed metric dimension throttling, whereas it was $2 \sqrt{2/3}$ for standard and edge metric dimension throttling.

\begin{thm}
    For all integers $n\geq 3$, 
    \begin{itemize}
    \item $\thrm(C_n) = 2\sqrt{n}(1\pm o(1))$, \item $\thrm(P_n) = 2\sqrt{n}(1\pm o(1))$.
    \end{itemize}
\end{thm}

\begin{proof}
    For the upper bound, we can use $\lceil \sqrt{n} \rceil$ landmarks spaced at intervals of approximately equal length. With a sensor that sees out to a distance of $\lceil \sqrt{n} \rceil$, we have $\thrm(C_n) \le 2\sqrt{n}(1\pm o(1))$. See Figure~\ref{fig:68} for an example of this construction when $n = 25$.

    For the lower bound, suppose that the sensor can see out to distance $r$. Then the distance between any two consecutive vertices $A$ and $B$ on the cycle with landmarks must be at most $r+1$, or else the vertex $A$ and the edge between $A$ and $B$  incident to $A$ would not be distinguished from each other. 

    Thus, we have at least $n/(r+1)$ landmarks, so the mixed metric dimension throttling number of $C_n$ is at least \[r+\frac{n}{r+1} \ge 2\sqrt{n}(1\pm o(1)).\]
    By similar reasoning as above, $\thrm(P_n) = 2\sqrt{n}(1\pm o(1))$.
\end{proof}

\begin{figure}[h]
    \centering
    \includegraphics[width=0.3\textwidth]{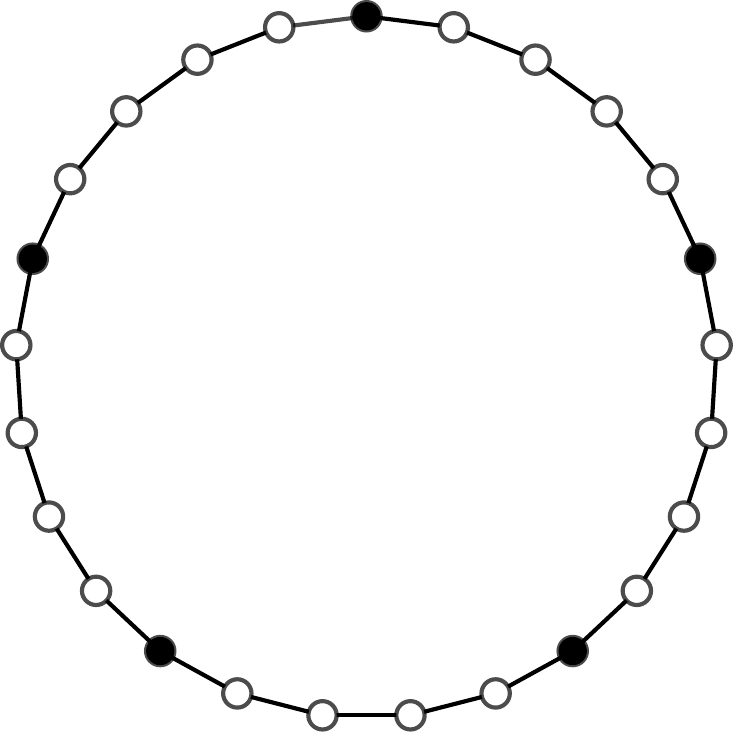}
    \caption{An optimal mixed metric dimension throttling configuration for $C_n$ with $n=25$, $r=5$, and the landmarks are at distance $5$.}
    \label{fig:68}
\end{figure}

We showed that the minimum possible standard metric dimension throttling number of a tree of order $n$ is $\Theta(n^{1/3})$, and we also obtained the same bound for edge metric dimension throttling. For mixed metric dimension on trees of order $n$, we prove that the minimum possible throttling number is $\Theta(n^{1/2})$. 

\begin{thm}\label{thm:xthr-tree}
The minimum possible mixed metric dimension throttling number of a tree of order $n$ is $\Theta(n^{1/2})$.
\end{thm}
\begin{proof}
The proof is very similar to the proof of Theorem~\ref{thm:thr-tree}. For the upper bound, use $P_n$. For the lower bound, suppose that $G$ is a tree of order $n$, $k=r+\mdim_r(G)$ for some integer $r\in[0,k]$, and $S$ is a distance-$r$ mixed resolving set for $G$ of size $\mdim_r(G)$. Consider the subgraph $G'$ of $G$ formed by all the paths between the vertices in $S$. Every vertex of degree $1$ in $G$ must be in every mixed resolving set for $G$ \cite{mmd}, and every distance-$r$ mixed resolving set for $G$ is also a mixed resolving set for $G$, so we have $G = G'$, which implies that $n = |V(G')|=O(k^2)$ and $k=\Omega(n^{1/2})$. 
\end{proof}



We conclude by showing that as with edge metric dimension throttling, the family of hypercube graphs have mixed metric dimension throttling number close to the lower bound.
Recall that by Theorem~\ref{thm:minthx}, for all graphs $G$ of order $n$ with $m$ edges, $\thrm(G) = \Omega\left(\frac{\log{(m+n)}}{\log{\log{(m+n)}}}\right)$.

\begin{prop}
    There exists an infinite family of connected graphs $G$ for which $\thrm(G) = O\left(\log{(m+n)}\right)$, where $G$ has order $n$ and $m = \Theta(n \log n)$ edges.
\end{prop}

\begin{proof}
    The hypercube graph $Q_d$ has order $n = 2^d$, $m = d 2^{d-1}$ edges, diameter $d$, and $\mdim(Q_d) = \Theta\left(\frac{d}{\log{d}}\right)$ \cite{mmd}. Thus, $\thrm(Q_d) = O(d) = O\paren{\log(m+n)}$.
\end{proof}

\section{Conclusion}\label{sec:conclusion}
In this paper, we introduced and studied throttling for metric dimension and its variants. Table 1 summarizes our main results. Below we highlight several of them in more detail and discuss directions for future work.

We proved the NP-hardness of standard, edge, and mixed metric dimension throttling. Our proofs of NP-hardness use disconnected graphs; it is a problem for future work to obtain alternate NP-hardness proofs when the problems are restricted to connected graphs.

We proved that the minimum possible metric dimension throttling number of any graph of order $n$ is $\Theta\left(\frac{\log{n}}{\log{\log{n}}}\right)$. For edge metric dimension, we showed that for all graphs $G$ with $m$ edges, $\thre(G) = \Omega\left(\frac{\log{m}}{\log{\log{m}}}\right)$. Moreover, we found an infinite family of connected graphs $G$ with $m$ edges for which $\thre(G) = O\left(\log{m}\right)$. There is a $\log{\log{m}}$ gap between these bounds. There is likewise a gap between our upper and lower bounds for the minimum possible mixed metric dimension throttling number. It remains an open problem to close these gaps and find the coefficients of the leading terms. 

For standard and edge metric dimension, we showed that the minimum possible throttling number of any tree of order $n$ is $\Theta\left(n^{1/3}\right)$, while for mixed metric dimension it is $\Theta\left(n^{1/2}\right)$. For paths and cycles, we found sharp bounds up to the leading term for standard, edge, and mixed metric dimension throttling. It is an open problem to find the exact values of the throttling numbers for these graphs.

\begin{center}
    
\begin{talltblr}[
  caption={Summary of main results}
]
{|p{28mm,m}|p{32mm,m}|p{32mm,m}|p{48mm,m}|}
\hline
     & $\thr$ & $\thre$ & $\thrm$ \\
\hline
Complexity & NP-Complete & NP-Complete & NP-Complete \\
\hline
$\min$ on graphs with $n$ vertices&
$\Theta\left(\frac{\log n}{\log \log n}\right)$& 0
& $\Omega\left(\frac{\log n}{\log \log n}\right)$ \\
\hline
$\min$ on graphs with $m$ edges & $\Theta\left(\frac{\log m}{\log \log m}\right)$
&
$\Omega\left(\frac{\log m}{\log \log m}\right)$ $O(\log m)$ & $\Omega\left(\frac{\log m}{\log \log m}\right)$\\
\hline
$\min$ on graphs with $n$ vertices and $m$ edges &
$\Omega\left(\frac{\log n}{\log \log n}\right)$ &
$\Omega\left(\frac{\log m}{\log \log m}\right)$ & $\Omega\left(\frac{\log (m+n)}{\log \log (m+n)}\right)$ \\
\hline
$\max$ on graphs of order $n$ & $n-1$ & $\begin{cases}
    0 & n=1,2\\
    n-1 & n>2
\end{cases}$ & $\begin{cases}
    0 & n=1\\
    n & n>1
\end{cases}$ \\
\hline
$\max$ on trees of order $n$ &
$n-1$ & 
$\begin{cases}
    0 & n=1\\
    n-2 & n>1
\end{cases}$ &
$\begin{cases}
    0 & n=1\\
    n & n>1
\end{cases}$
\\
\hline
$\min$ on trees of order $n$ & $\Theta(n^{1/3})$ & $\Theta(n^{1/3})$ & $\Theta(n^{1/2})$\\
\hline
$K_n$ & $n-1$ & $\begin{cases}
   0 & n=1,2\\
   n-1& n>2
\end{cases}$ & $\begin{cases}
    0 & n=1\\
    n & n > 1
\end{cases}$\\
\hline
subtree-monotone & yes & yes & yes \\
\hline
$K_{s,t}$ & $s+t-1$ & $s+t-2$ & $\begin{cases}
    s+t&\min(s,t)\le2\\
    s+t-1&\text{otherwise}
\end{cases}$\\
\hline
$P_n$,\, $C_n$ & $2\sqrt{\frac{2}{3}n}\left(1\pm o(1)\right)$ & $2\sqrt{\frac{2}{3}n}\left(1\pm o(1)\right)$&
$2\sqrt{n}\left(1\pm o(1)\right)$\\
\hline
\end{talltblr}
\end{center}

\end{document}